\documentclass[11pt]{article}

\usepackage[margin=1cm]{geometry}
\usepackage{fullpage}
\usepackage{amsmath}
\usepackage{amsthm}
\usepackage{amssymb}
\usepackage[hidelinks]{hyperref}
\usepackage{cleveref}
\usepackage{textcomp}
\usepackage{xcolor}

\newcommand{\mbb}[1]{\mathbb{#1}}
\newcommand{\mbf}[1]{\mathbf{#1}}
\newcommand{\mc}[1]{\mathcal{#1}}

\newcommand{\tr}{\textup{tr}}
\newcommand{\Tr}[1]{\left\langle#1\right\rangle}
\newcommand{\wt}[1]{\widetilde{#1}}
\renewcommand{\det}{\textup{det}}
\renewcommand{\Re}{\textup{Re}}
\renewcommand{\Im}{\textup{Im}}

\numberwithin{equation}{section}

\newtheorem{theorem}{Theorem}[section]

\newtheorem{lemma}{Lemma}[section]
\newtheorem{cor}{Corollary}[section]
\newtheorem{prop}{Proposition}[section]

\theoremstyle{remark}

\theoremstyle{definition}

\title{Bulk Universality for Complex Non-Hermitian Matrices with Independent and Identically Distributed Entries}
\author{Anna Maltsev and Mohammed Osman\footnote{mohammed.osman@qmul.ac.uk}}
\date{\small Queen Mary, University of London}
\begin{document}

\maketitle

\abstract{We consider $N\times N$ matrices with complex entries that are perturbed by a complex Gaussian matrix with small variance. We prove that if the unperturbed matrix satisfies certain local laws then the bulk correlation functions are universal in the large $N$ limit. Assuming the entries are independent and identically distributed (iid) with a common distribution that has finite moments, the Gaussian component is removed by the four moment theorem of Tao and Vu.}

\section{Introduction}
The concept of universality in random matrix theory is the idea that various spectral properties of random matrices do not depend on the particulars of the randomness of the matrix entries. Universality is what allows random matrix theory to be applicable in a wide variety of contexts. Two well-known examples are in data analytics (e.g., Principal Component Analysis), where data and noise come from some unspecified distribution, and in nuclear physics where the randomness models the complexity of a large atomic nucleus. Mathematically the Wigner-Mehta-Dyson Universality Conjecture was formulated in Mehta's book in 1990 for Hermitian matrices \cite{mehta_random_1990}. This conjecture is akin to the Central Limit Theorem in statistics as it states that the limiting properties of some aggregate of a large number of random variables are independent of the distribution of the latter.

While Mehta found non-Hermitian random matrices to be of ``no immediate physical interest" {\cite[Chapter 13]{mehta_random_1990}}, it is now understood that such matrices model open chaotic systems. Additionally, with the recent rise of large data and non-symmetric networks, e.g., the explicit mapping of neuronal networks such as in \cite{cook_whole-animal_2019}, one expects that non-Hermitian random matrix models will pave the way for our further understanding in applied fields. As real-world data and networks exhibit complexity and randomness with unknown non-Gaussian properties, it is the universality of non-Hermitian random ensemble's spectral properties that will allow this mathematical theory to be relevant.

Proving universality in non-Hermitian ensembles has proved to be a challenge, and our understanding of non-Hermitian universality lags behind its Hermitian counterpart. The main difficulties are the instability of the eigenvalues under perturbation and the lack of unitarity in the eigenvectors. While the first proof in the non-Hermitian analogue of the Wigner ensemble (i.e. the random matrix whose entries are i.i.d. random variables) appeared in the paper of Tao and Vu \cite{tao_random_2015} soon after proofs of universality in Hermitian ensembles, they require the first four moments of the given distribution to match those of the Gaussian. In the Hermitian case, an analogous four-moment theorem could be combined with a proof of universality for matrices with a small Gaussian component in order to remove the moment-matching condition. We use the same idea to prove universality for non-Hermitian matrices with i.i.d. complex-valued entries from any distribution that has finite moments.

To be precise, in the Hermitian case, the well-established ``3-step strategy", due to \cite{erdos_bulk_2010} and by which universality has been proven for very general classes of matrices, consists of three steps:
\begin{enumerate}
\item a local law for the trace of the resolvent $G\left(z\right)=(W-z)^{-1}$ for $|\Im z|\geq N^{-1+\epsilon}$;
\item universality for Gauss-divisible matrices, i.e. matrices of the form $W_{t}=W+\sqrt{t}V$, where $W$ is a matrix satisfying the local law, $V$ is Gaussian and $t$ is small;
\item comparison between the statistics of $W_{t}$ and $W$.
\end{enumerate}
It is natural to ask whether the three-step strategy can be adapted to prove universality for non-Hermitian matrices. The goal of this paper is to give an affirmative answer in the case of matrices with i.i.d. complex entries. In fact, Steps 1 and 3 have already been adapted: for the former there are the local laws of Bourgade, Yau and Yin \cite{bourgade_local_2014} and Alt, Erd\H{o}s and Kr\"{u}ger \cite{alt_local_2018,alt_spectral_2021}; for the latter there is the four moment theorem of Tao and Vu \cite{tao_random_2015}. We also note that the method of time reversal from \cite{erdos_bulk_2010} immediately generalises to non-Hermitian matrices, since it is based on an approximation in the joint distribution of all matrix entries.

We are therefore concerned with the second step, namely universality for Gauss-divisible non-Hermitian matrices. For complex Hermitian matrices, the first work in this direction was by Johansson \cite{johansson_universality_2001}, who performed an asymptotic analysis of an explicit integral formula for the correlation function of Gauss-divisible matrices. Central to the derivation of this formula is the Harish-Chandra--Itzykson--Zuber (HCIZ) integral, which allows one to integrate out the eigenvectors. It turns out that by a different approach one can derive an explicit integral formula for non-Hermitian matrices that, although slightly more complicated, is still amenable to asymptotic analysis. In place of the HCIZ integral we use a combination of a change of variables and supersymmetry. By performing $k$ successive Householder transformations, we obtain a block upper triangular matrix whose first $k$ diagonal elements are any given $k$-tuple of eigenvalues. We can then relate the correlation function of $k$ eigenvalues to an integral over the remaining variables of this transformation. The Jacobian of this transformation leads to the expectation of a product of characteristic polynomials, for which we use supersymmetry to obtain an alternative integral representation that can be analysed by Laplace's method.

The idea of using a partial Schur decomposition to calculate eigenvalue densities goes back to Edelman, Kostlan and Shub \cite{edelman_how_1994}, who used a single Householder transformation to calculate the expected number of real eigenvalues for a real Gaussian matrix. It has also been used by Fyodorov and Khoruzhenko \cite{fyodorov_absolute_2007} to calculate the density of the imaginary parts of eigenvalues of rank one anti-Hermitian perturbations of the GUE, and by Fyodorov \cite{fyodorov_statistics_2018} to calculate the density of eigenvector overlaps in the Ginibre ensemble. The partial Schur decomposition with $k>1$ Householder transformations was used in the thesis of the second author to calculate the correlation functions of low rank anti-Hermitian perturbations of the GUE \cite{osman_thesis_2022}. On the other hand, a full Schur decomposition was used by Liu and Zhang \cite{liu_phase_2022} to study finite rank deformations of the Ginibre ensemble.  In eq. (1.9) of Proposition 1.3 in their paper, they obtain in this way an integral formula for the $k$-point correlation functions of matrices $A+B$, where $A$ is a fixed matrix of rank $r\leq N-k$ and $B$ is a Ginibre matrix. By partial Schur decomposition we obtain a different formula (see \eqref{eq:kpointint} below) that holds without restriction on the rank of $A$.

The most general approach to universality for Gauss-divisible Hermitian matrices, due to Erd\H{o}s, Schlein and Yau \cite{erdos_universality_2011}, goes via Dyson Brownian Motion, i.e. the system of stochastic differential equations (SDEs) governing the evolution of the eigenvalues of $W_{t}$ in time $t$. For non-Hermitian matrices, the eigenvalues do not satisfy a closed system of SDEs but instead are coupled with the inner products between right and left eigenvectors \cite{grela_full_2018}. The analysis of the whole system of equations is therefore much more involved and we do not pursue this route.

We note that the universality of the correlation functions at the edge (i.e. the unit circle) has been established by Cipolloni, Erd\H{o}s and Schr\"{o}der \cite{cipolloni_edge_2021}. Their result holds for both real and complex matrices without the need for matching moments. Therefore in this paper we are concerned only with correlations in the bulk.

\section{Main Results}
To state our results, we introduce some notation and definitions. The set of $n\times n$ complex (respectively complex Hermitian) is denoted by $\mbf{M}_{n}$ (respectively $\mbf{M}^{sa}_{n}$). {The unitary group of dimension $n$ is denoted by $\mbf{U}_{n}$. The $n\times n$ identity matrix is denoted by $1_{n}$, although the subscript is sometimes suppressed when there is no ambiguity (i.e. when it appears in an expression with another matrix whose dimension is known).} Let  $\|A\|$ denote the operator norm, $\left\langle A\right\rangle=N^{-1}\tr A$ and $|A|=\sqrt{AA^{*}}$. For $\mbf{v}\in\mbb{C}^{N}$,\,$A\in\mbf{M}_{N}$ we denote by $d\mbf{v},\,dA$ the respective Lebesgue measures.

As is standard in the analysis of non-Hermitian random matrices, we define the Hermitisation $\mc{H}_{z}$ of $A$ by
\begin{align}\label{eq:hermitisation}
    \mc{H}_{z}&:=\begin{pmatrix}0&A-z\\A^{*}-\bar{z}&0\end{pmatrix},
\end{align}
for some $z\in\mbb{C}$. Here and in the following we denote scalar multiples of the identity by the scalar alone. Define the block matrices
\begin{align}
    E&:=\begin{pmatrix}0&1\\0&0\end{pmatrix}\otimes1_{N},\label{eq:rho}\\
    E_{+}&:=\begin{pmatrix}1&0\\0&0\end{pmatrix}\otimes1_{N},\label{eq:rhoPlus}\\
    E_{-}&:=\begin{pmatrix}0&0\\0&1\end{pmatrix}\otimes1_{N}.\label{eq:rhoMinus}
\end{align}
{Alongside} the resolvent $G_{z}\left(\eta\right)=(\mc{H}_{z}-i\eta)^{-1}$, we also define the resolvents
\begin{align}
    H_{z}\left(\eta\right):=\left(\eta^{2}+(A-z)(A-z)^{*}\right)^{-1},\label{eq:G}\\
    \wt{H}_{z}\left(\eta\right):=\left(\eta^{2}+(A-z)^{*}(A-z)\right)^{-1}\label{eq:H},
\end{align}
which appear in the block matrix representation of $ G_{z}$:
\begin{align}
     G_{z}\left(\eta\right)&=\begin{pmatrix}i\eta H_{z}\left(\eta\right)&H_{z}\left(\eta\right)(A-z)\\(A-z)^{*}H_{z}\left(\eta\right)&i\eta \wt{H}_{z}\left(\eta\right)\end{pmatrix}.
\end{align}
The following traces will be needed in the sequel:
\begin{align}
    g_{z}\left(\eta\right)&:=\eta\left\langle H_{z}\left(\eta\right)\right\rangle,\label{eq:g}\\
    \alpha_{z}\left(\eta\right)&:=\eta^{2}\left\langle H_{z}\left(\eta\right)\wt{H}_{z}\left(\eta\right)\right\rangle,\label{eq:alpha}\\
    \beta_{z}\left(\eta\right)&:=\eta\left\langle H^{2}_{z}\left(\eta\right)(A-z)\right\rangle,\label{eq:beta}\\
    \gamma_{z}\left(\eta\right)&:=\eta^{2}\left\langle H^{2}_{z}\left(\eta\right)\right\rangle.\label{eq:gamma}
\end{align}
In terms of $ G_{z}$ we have
\begin{align*}
    g_{z}\left(\eta\right)&=\frac{1}{2}\Im\left\langle G_{z}\left(\eta\right)\right\rangle,\\
    \alpha_{z}\left(\eta\right)&=-\left\langle G_{z}\left(\eta\right)E G_{z}\left(\eta\right)E^{*}\right\rangle,\\
    \beta_{z}\left(\eta\right)&=\frac{1}{2i}\left\langle G^{2}_{z}\left(\eta\right)E^{*}\right\rangle,\\
    \gamma_{z}\left(\eta\right)&=\frac{1}{4i\eta}\left\langle G_{z}\left(\eta\right)(1-i\eta G_{z}\left(\eta\right))\right\rangle.
\end{align*}
These quantities will appear in the following combination:
\begin{align}
    \sigma_{z}\left(\eta\right)&:=\alpha_{z}(\eta)+\frac{|\beta_{z}(\eta)|^{2}}{\gamma_{z}(\eta)}\label{eq:sigma}.
\end{align}
Note that
\begin{align*}
    \alpha_{z}(\eta)&=\frac{1}{4}\Delta_{z}\log|\det(\mc{H}_{z}-i\eta)|
\end{align*}
is a smoothed version of the eigenvalue density on a scale $\eta$.

Now we come to the assumptions we make on $A$. {Fix $\epsilon>0$. Then for any $\omega>0$ there are constants $c,C>0$ depending on $\epsilon$ and $\omega$ but not on $N$ such that for all $z\in\mbb{C}$ with $|z|<1-\omega$ and all $\eta>N^{-1/2+\epsilon}$:} 
\paragraph{A1:} We have
\begin{align}
    \tag{A1.1}
    c\leq g_{z}\left(\eta\right)&\leq C,\label{eq:A_g}\\
    \tag{A1.2}
    c\leq\alpha_{z}\left(\eta\right)&\leq C,\label{eq:A_alpha}\\
    \tag{A1.3}
    |\beta_{z}\left(\eta\right)|&\leq C,\label{eq:A_beta}\\
    \tag{A1.4}
    c\leq\eta\gamma_{z}\left(\eta\right)&\leq C.\label{eq:A_gamma}
\end{align}
\paragraph{A2:} For $B_{i}\in\{E,E^{*}\}$ and $\eta_{i}\geq\eta$,
\begin{align}
    \tag{A2}
    \left|\left\langle G_{z}(\eta_{1})B_{1}G_{z}(\eta_{2})B_{2}\right\rangle\right|&\leq C.\label{eq:A2}
\end{align}
\paragraph{A3:} For $B_{i}\in\{E,E^{*}\}$ and $\eta_{i}\geq\eta$,
\begin{align}
    \tag{A3}
    \left|\left\langle G_{z}(\eta_{1})B_{1}G_{z}(\eta_{2})B_{2}G_{z}(\eta_{3})B_{3}\right\rangle\right|&\leq\frac{C}{\eta}.\label{eq:A3}
\end{align}

Gauss-divisible matrices are defined by
\begin{align}
    M_{t}&=A+\sqrt{t}B,\label{eq:GaussDivisible}
\end{align}
where $B$ is a complex Ginibre matrix, i.e. a matrix with i.i.d. complex Gaussian entries with mean zero and variance $1/N$. Recall that the $k$-point correlation functions of a point process $\{z_{1},z_{2},...\}$ are defined by
\begin{align}
    \mbb{E}\left[\sum_{i_{1}\neq\cdots\neq i_{k}}f(z_{i_{1}},...,z_{i_{k}})\right]&=:\int_{\mbb{C}^{k}}f(\mbf{z})\rho^{(k)}(\mbf{z})d\mbf{z},\label{eq:kpointdef}
\end{align}
for all smooth and compactly supported $f$ {(in the singular case one replaces $\rho^{(k)}(\mbf{z})d\mbf{z}$ with a correlation measure $d\rho^{(k)}(\mbf{z})$)}. For deterministic $A$, the correlation functions of $M_{t}$ are denoted by $\rho^{\left(k\right)}_{t}\left(\mbf{z};A\right)$. The limiting bulk correlation functions of the complex Ginibre ensemble are denoted by $\rho^{\left(k\right)}_{GinUE}$ and given by {\cite{ginibre_statistical_1965}}:
\begin{align}
    \rho^{\left(k\right)}_{GinUE}\left(\mbf{z}\right)&:=\det\left[\frac{1}{\pi}e^{-\frac{1}{2}\left(|z_{j}|^{2}+|z_{l}|^{2}\right)+\bar{z}_{j}z_{l}}\right].
\end{align}
The first result concerns the pointwise limit of $\rho^{\left(k\right)}_{t}$.
\begin{theorem}\label{thm1}
Let $\epsilon>0$ and $t\geq N^{-1/3+\epsilon}$. Let $A$ satisfy \eqref{eq:A_g}, \eqref{eq:A_alpha}, \eqref{eq:A_beta}, \eqref{eq:A_gamma} and \eqref{eq:A2}, $B$ be a complex Ginibre matrix, and $M_{t}=A+\sqrt{t}B$. Then for any {fixed} $z\in\mbb{C}$ such that $|z|<1$, there is a unique $\eta_{z,t}>0$ such that
\begin{align}
    t\langle H_{z}\left(\eta_{z,t}\right)\rangle&=1.
\end{align}
Let $\rho^{(k)}_{t}$ denote the $k$-point correlation function of $M_{t}$. Then
\begin{align}\label{eq:thm1}
    \lim_{N\to\infty}\frac{1}{\left(N\sigma_{z,t}\right)^{k}}\rho^{\left(k\right)}_{t}\left(z+\frac{\mbf{z}}{\sqrt{N\sigma_{z,t}}};A\right)&=\rho^{\left(k\right)}_{GinUE}\left(\mbf{z}\right),
\end{align}
uniformly on compact subsets of $\mbb{C}^{k}$, where $\sigma_{z,t}:=\sigma_{z}(\eta_{z,t})$. If in addition $A$ satisfies \eqref{eq:A3} then the conclusion holds for $t\geq N^{-1/2+\epsilon}$.
\end{theorem}
We remark that the assumptions \eqref{eq:A_alpha}, \eqref{eq:A_beta} and \eqref{eq:A_gamma} imply that $\sigma_{z,t}=\alpha_{z,t}+O(t)$, and since $\alpha_{z,t}$ approximates the eigenvalue density on a scale $t$ the rescaling by $\sigma_{z,t}$ in \Cref{thm1} is natural.

Using this theorem, the local law and the four moment theorem we can remove the Gaussian component and obtain the following.
\begin{theorem}\label{thm2}
{Let $\xi_{ij},\, i,j=1,...,N$ be i.i.d. random variables with zero mean, unit variance and finite moments, and $A=\frac{1}{\sqrt{N}}(\xi_{ij})_{i,j=1}^{N}$.} Then for any $f\in C_{c}^{\infty}\left(\mbb{C}^{k}\right)$ and {fixed} $z\in\mbb{C}$ such that $|z|<1$ we have
\begin{align}\label{eq:thm2}
    \lim_{N\to\infty}\mbb{E}\left[\sum_{i_{1}\neq\cdots\neq i_{k}}f\left(\sqrt{N}(z_{i_{1}}-z),...,\sqrt{N}(z_{i_{k}}-z)\right)\right]&=\int_{\mbb{C}^{k}}f\left(\mbf{z}\right)\rho^{\left(k\right)}_{GinUE}\left(\mbf{z}\right)d\mbf{z},
\end{align}
where $z_{i}$ are the eigenvalues of $A$.
\end{theorem}

Let us end this section by a discussion of the optimal size of the variance $t$ of the Gaussian component. In the course of the proof of \Cref{thm1}, we find that we need to estimate traces of alternating products of $G_{z}\left(\eta\right)$ with matrices whose diagonal blocks are zero. In this case a recent result of Cipolloni, Erd\H{o}s, Henheik and Schr\"{o}der \cite{cipolloni_optimal_2023} shows that, for traces involving up to two resolvents, the error is smaller by a factor of $\sqrt{\eta}$ for each such off-diagonal matrix. In particular, for $\left\langle G_{z}B_{1}G_{z}B_{2}\right\rangle$, where the diagonal blocks of $B_{1}$ and $B_{2}$ are zero, the error is $O\left(1/N\eta\right)$. This reduction in the size of the error is analogous to the reduction that occurs when taking traces of resolvents of Wigner matrices alternating with traceless matrices that was shown in \cite{cipolloni_thermalisation_2022}. If this can be extended to traces of finite products $\langle\prod_{m=1}^{n}G_{z}(\eta_{m})B_{m}\rangle$, we might expect that the optimal $t$ is in fact $N^{-1+\epsilon}$, as in the Hermitian case. However, as part of the proof we need to transfer local laws from $A$ to a projection of $A$ onto a certain random subspace of codimension $k$, and in doing so we make use of the assumption $Nt^{2}\gg1$. Whether this can be relaxed is left for future work.

\section{Preliminaries}

We collect some lemmas that we will use in the rest of the paper. The following lemma relates the resolvent of a minor to that of the original matrix.
\begin{lemma}\label{lem:minorresolvent}
Let $A\in \mbf{M}_{N}$ {be non-singular}; $U=\left(U_{k},U_{N-k}\right)\in\mbf{U}_{N}$, where $U_{k}$ and $U_{N-k}$ are $N\times k$ and $N\times\left(N-k\right)$ respectively; and $B=U_{N-k}^{*}AU_{N-k}$ {be non-singular}. Then
\begin{align}\label{eq:minorresolvent}
    U\begin{pmatrix}0&0\\0&B^{-1}\end{pmatrix}U^{*}&=A^{-1}-A^{-1}U_{k}\left(U_{k}^{*}A^{-1}U_{k}\right)^{-1}U_{k}^{*}A^{-1},
\end{align}
and, if $\Re A>0$,
\begin{align}\label{eq:minornorm}
    \left\|A^{-1}U_{k}\left(U_{k}^{*}A^{-1}U_{k}\right)^{-1}U_{k}^{*}A^{-1}\right\|&\leq\left\|\left(\Re A\right)^{-1}\right\|.
\end{align}
If $\Im A>0$ \eqref{eq:minornorm} holds with $\left\|\left(\Im A\right)^{-1}\right\|$ on the right hand side.
\end{lemma}
\begin{proof}
For the first statement we follow the argument of the blog post \cite{tao_inverting_2017}. Let
\begin{align*}
    U^{*}AU&=\begin{pmatrix}X&Y\\Z&B\end{pmatrix};
\end{align*}
for sufficiently large $t>0$ we have
\begin{align*}
    U^{*}\left(A+tU_{k}U_{k}^{*}\right)^{-1}U&=\begin{pmatrix}t+X&Y\\Z&B\end{pmatrix}^{-1}\\
    &=\begin{pmatrix}\left(t+X-YB^{-1}Z\right)^{-1}&-\left(t+X-YB^{-1}Z\right)^{-1}YB^{-1}\\
    -(B+Z\left(X+t\right)^{-1}Y)^{-1}Z\left(X+t\right)^{-1}&(B+Z\left(X+t\right)^{-1}Y)^{-1}\end{pmatrix}\\
    &\to\begin{pmatrix}0&0\\0&B^{-1}\end{pmatrix}\quad\text{as}\quad t\to\infty.
\end{align*}
On the other hand, by the Woodbury matrix identity we have
\begin{align*}
    \left(A+tU_{k}U_{k}^{*}\right)^{-1}&=A^{-1}-A^{-1}U_{k}\left(t^{-1}+U_{k}^{*}A^{-1}U_{k}\right)^{-1}U_{k}^{*}A^{-1}\\
    &\to A^{-1}-A^{-1}U_{k}\left(U_{k}^{*}A^{-1}U_{k}\right)^{-1}U_{k}^{*}A^{-1}\quad\text{as}\quad t\to\infty.
\end{align*}
Equating the two expressions we get \eqref{eq:minorresolvent}.

Now assume $A=B+iC$ with $B>0$. Let
\begin{align*}
    V_{k}&=B^{\frac{1}{2}}\left(B+iC\right)^{-1}U_{k},\\
    \widetilde{V}_{k}&=V_{k}\left(V_{k}^{*}V_{k}\right)^{-\frac{1}{2}},\\
    H&=B^{-\frac{1}{2}}CB^{-\frac{1}{2}},\\
    W&=\left(1-iH\right)(1+iH)^{-1};
\end{align*}
note that since $B>0$, $V_{k}^{*}V_{k}=U_{k}^{*}\left(B-iC\right)^{-1}B\left(B+iC\right)^{-1}U_{k}>0$ and hence $\widetilde{V}_{k}$ is well-defined. Moreover, $\widetilde{V}_{k}^{*}\widetilde{V}_{k}=1_{k}$, $H$ is Hermitian and $W$ is unitary. Denoting by $r$ the left hand side of \eqref{eq:minornorm}, we have
\begin{align*}
    r&=\left\|B^{-\frac{1}{2}}\tilde{V}_{k}\left(1-i\tilde{V}_{k}^{*}H\tilde{V}_{k}\right)^{-1}\tilde{V}_{k}^{*}WB^{-\frac{1}{2}}\right\|\\
    &\leq\left\|B^{-1}\right\|\cdot\left\|\left(1-i\tilde{V}_{k}^{*}H\tilde{V}_{k}\right)^{-1}\right\|\cdot\left\|\tilde{V}_{k}\right\|\cdot\left\|W\right\|\\
    &\leq\left\|B^{-1}\right\|.
\end{align*}
If $C>0$, we repeat the above with $-iA$ in place of $A$.
\end{proof}
When $k=1$, $U=\mbf{u}$ is a vector on the unit sphere, and the lemma reduces to
\begin{align}
    \Big|\frac{\mbf{u}^{*}A^{-1}BA^{-1}\mbf{u}}{\mbf{u}^{*}A^{-1}\mbf{u}}\Big|&\leq\left\|\left(\Re A\right)^{-1}\right\|.
\end{align}
{We will use often this lemma to bound ratios of determinants as in the following.
\begin{lemma}\label{lem:detRatio}
Let $V\in\mbb{C}^{k\times n}$ and $Y,W\in\mbf{M}_{n}$ with $Y,W$ invertible and $\Im Y>0$. Assume that
\begin{align}
    \rho&:=\left\|(\Im Y)^{-1}\right\|\cdot\left\|Y(Y^{-1}-W^{-1})Y\right\|<1.
\end{align}
Then we have
\begin{align}
    \left|\frac{\det V^{*}W^{-1}V}{\det V^{*}Y^{-1}V}-1\right|&\leq\frac{k\rho}{1-\rho}e^{\frac{k\rho}{1-\rho}}.\label{eq:detRatio}
\end{align}
\end{lemma}
\begin{proof}
Using the identity $\det(1+AB)=\det(1+BA)$ we rewrite the ratio of determinants as
\begin{align*}
    \frac{\det V^{*}W^{-1}V}{\det V^{*}Y^{-1}V}&=\det\left[1_{k}-(V^{*}Y^{-1}V)^{-1}V^{*}(Y^{-1}-W^{-1})V\right]\\
    &=\det\left(1_{n}-P\right),
\end{align*}
where
\begin{align*}
    P&:=Y^{-1}V(V^{*}Y^{-1}V)^{-1}V^{*}Y^{-1}\cdot Y(Y^{-1}-W^{-1})Y.
\end{align*}
Using \eqref{eq:minornorm} we have
\begin{align*}
    \left\|Y^{-1}V(V^{*}Y^{-1}V)^{-1}V^{*}Y^{-1}\right\|&\leq\|(\Im Y)^{-1}\|,
\end{align*}
which gives $\|P\|\leq\rho<1$. Thus $1-uP$ is invertible for $u\in[0,1]$ and from the integral representation of the logarithm we obtain
\begin{align*}
    \frac{\det V^{*}W^{-1}V}{\det V^{*}Y^{-1}V}&=\exp\left\{-\int_{0}^{1}\tr((1-uP)^{-1}P)du\right\}.
\end{align*}
Since $\text{rank}(P)=k$, we have
\begin{align*}
    \left|\int_{0}^{1}\tr((1-uP)^{-1}P)du\right|&\leq\frac{k\rho}{1-\rho}.
\end{align*}
Thus using $|e^{x}-1|\leq xe^{x}$ we obtain \eqref{eq:detRatio}.
\end{proof}
}

The next lemma is known as Fischer's inequality \cite{fischer_uber_1908}.
\begin{lemma}\label{lem:fischerineq}
Let $A$ be a positive semi-definite block matrix with diagonal blocks $A_{nn}$. Then
\begin{align}\label{eq:fischerineq}
    \det A&\leq\prod_{n}\det A_{nn}.
\end{align}
\end{lemma}

The next lemma deals with integrals over the Stiefel manifold $\mbf{U}_{N,k}:=\mbf{U}_{N}/\mbf{U}_{N-k}$ {with respect to the Haar measure $d\mu_{N,k}$ with total mass equal to $\text{Vol}(\mbf{U}_{N,k})$}. We state it for general $k$, although we will only use it for $k=1$, i.e. for integrals on the sphere.
\begin{lemma}\label{lem:sphericalint}
Let $f\in L^{1}\left(\mbb{C}^{N\times k}\right)$ be continuous on a neighbourhood of $\mbf{U}_{N,k}$ and define $\hat{f}: \mbf{M}^{sa}_{k}\to\mbb{C}$ by
\begin{align}
    \hat{f}(H)&:=\frac{1}{\pi^{k^{2}}}\int_{\mbb{C}^{N\times k}}e^{-i\tr HX^{*}X}f\left(X\right)dX.
\end{align}
Assume that $\hat{f}\in L^{1}\left( M^{sa}_{k}\right)$; then
\begin{align}\label{eq:sphericalint}
    \int_{\mbf{U}_{N,k}}f\left(U\right)d\mu_{N,k}\left(U\right)&=\int_{ M^{sa}_{k}}e^{i\tr H}\hat{f}(H)dH.
\end{align}
\end{lemma}
\begin{proof}
Let
\begin{align}
    I_{\epsilon}&=\frac{1}{2^{\frac{k}{2}}\left(\pi\epsilon\right)^{\frac{k^{2}}{2}}}\int_{\mbb{C}^{N\times k}}e^{-\frac{1}{2\epsilon}\tr\left(X^{*}X-1_{k}\right)^{2}}f\left(X\right)dX.
\end{align}
Any $X\in\mbb{C}^{N\times k}$ has a polar decomposition $X=UP^{\frac{1}{2}}$, where $U\in\mbf{U}_{N,k}$ and $P\geq0$. The Jacobian of this transformation is (see e.g. \cite{diaz-garcia_wishart_2011}, Proposition 4)
\begin{align*}
    dX&=2^{-k}\det^{N-k}\left(P\right)dPd\mu_{N,k}\left(U\right).
\end{align*}
Thus we have
\begin{align*}
    I_{\epsilon}&=\frac{1}{2^{\frac{k}{2}}\left(\pi\epsilon\right)^{\frac{k^{2}}{2}}}\int_{P\geq0}e^{-\frac{1}{2\epsilon}\tr\left(P-1_{k}\right)^{2}}g\left(P\right)dP,
\end{align*}
where
\begin{align*}
    g\left(P\right)&=\frac{1}{2^{k}}\det^{N-k}\left(P\right)\int_{\mbf{U}_{N,k}}f\left(UP^{\frac{1}{2}}\right)d\mu_{N,k}\left(U\right).
\end{align*}
Since $g\in L^{1}$ and is continuous at $P=1$ we can take the limit $\epsilon\to0$ to obtain
\begin{align}
    \lim_{\epsilon\to0}I_{\epsilon}&=g\left(1_{k}\right)=\frac{1}{2^{k}}\int_{\mbf{U}_{N,k}}f\left(U\right)d\mu_{N,k}\left(U\right).\label{eq:Ieps1}
\end{align}
On the other hand, by linearising the quadratic exponent through the introduction of an auxilliary variable (i.e. a Hubbard-Stratonovich transformation), we have
\begin{align*}
    I_{\epsilon}&=\frac{1}{2^{k}\pi^{k^{2}}}\int_{\mbb{C}^{N\times k}}\int_{ M^{sa}_{k}}e^{-\frac{\epsilon}{2}\tr H^{2}+i\tr H\left(1_{k}-X^{*}X\right)}f(X)dHdX\\
    &=\frac{1}{2^{k}}\int_{ M^{sa}_{k}}e^{-\frac{\epsilon}{2}\tr H^{2}+i\tr H}\hat{f}(H)dH.
\end{align*}
In the second line we have interchanged the $H$ and $X$ integrals, which is justified by the assumption $f\in L^{1}$ and the factor $e^{-\frac{\epsilon}{2}\tr H^{2}}$. Since moreover $\hat{f}\in L^{1}$, we can take the limit $\epsilon\to0$ inside the integral:
\begin{align}
    \lim_{\epsilon\to0}I_{\epsilon}&=\frac{1}{2^{k}}\int_{ M^{sa}_{k}}e^{i\tr H}\hat{f}(H)dH.\label{eq:Ieps2}
\end{align}
Comparing \eqref{eq:Ieps1} and \eqref{eq:Ieps2} we obtain \eqref{eq:sphericalint}.
\end{proof}
The requirement that $f\in L^{1}\left(\mbb{C}^{N\times k}\right)$ can be relaxed by noting that on the Stiefel manifold we have $U^{*}U=1_{k}$ and hence
\begin{align*}
    f\left(U\right)&=e^{ka-a\tr U^{*}U}f\left(U\right)=:e^{ka}f_{a}\left(U\right),
\end{align*}
for any $a>0$. If $|f(X)|\leq e^{c\tr X^{*}X}$ for some $c>0$, then $f_{a}\in L^{1}\left(\mbb{C}^{N\times k}\right)$ for $a>c$. If in addition $\hat{f}_{a}\in L^{1}\left( M^{sa}_{k}\right)$, we can apply the lemma to obtain
\begin{align}
    \int_{\mbf{U}_{N,k}}f\left(U\right)d\mu_{N,k}\left(U\right)&=e^{ka}\int_{ M^{sa}_{k}}e^{i\tr H}\hat{f}_{a}\left(H\right)dH.
\end{align}
When $k=1$, this lemma gives us a formula for integrals over the {complex unit sphere $S^{N-1}$}:
\begin{align}
    \int_{S^{N-1}}f\left(\mbf{u}\right)d\mu_{N,1}\left(\mbf{u}\right)&=\int_{-\infty}^{\infty}e^{ix}\hat{f}(x)dx.
\end{align}
Since $\hat{f}$ involves an integral over the whole vector space $\mbb{C}^{N}$, it is typically easier to evaluate than the original integral over $S^{N-1}$. We can then reduce an integral over $S^{N-1}$ to an integral over a single real variable, which is useful when taking the large $N$ limit.

We end this section by listing the relevant properties of matrices $A$ satisfying \eqref{eq:A_g}, \eqref{eq:A_alpha}, \eqref{eq:A_beta}, \eqref{eq:A_gamma} and \eqref{eq:A2}.
\begin{lemma}
Let $A$ satisfy \eqref{eq:A_g}, \eqref{eq:A_alpha}, \eqref{eq:A_beta}, \eqref{eq:A_gamma} and \eqref{eq:A2}, and let $z\in\mbb{C}$ such that $|z|<1$. For $\eta_{m}\geq\eta$ and $B_{m}\in\{E,E^{*}\}$ we have
\begin{align}
    \left|\left\langle\prod_{m=1}^{n}G_{z}\left(\eta_{m}\right)B_{m}\right\rangle\right|&\leq\begin{cases}C&\quad n=2\\C\eta^{-3/2}&\quad n=3\\ C\eta^{{2-n}}&\quad n\geq4\end{cases}.\label{eq:trace2}
\end{align}
\end{lemma}
\begin{proof}
Let $G_{i}=G_{z}\left(\eta_{i}\right)$. The case $n=2$ of \eqref{eq:trace2} follows directly from the assumption \eqref{eq:A2}. For the remaining cases we use Cauchy-Schwarz first and then \eqref{eq:A2}. When $n=3$ we have
\begin{align*}
    \left|\left\langle G_{1}B_{1}G_{2}B_{2}G_{3}B_{3}\right\rangle\right|&\leq\left\langle\frac{\Im G_{1}}{\eta_{1}}B_{1}\frac{\Im G_{2}}{\eta_{2}}B_{1}^{*}\right\rangle^{1/2}\left\langle B_{2}^{*}B_{2}G_{3}B_{3}B_{3}^{*}G^{*}_{3}\right\rangle^{1/2}\\
    &\leq\frac{\|B_{2}\|\cdot\|B_{3}\|}{\eta^{3/2}}\left\langle\Im\left(G_{1}\right)B_{1}\Im\left(G_{2}\right)B_{2}\right\rangle^{1/2}\left\langle \Im G_{2}\right\rangle^{1/2}\\
    &\leq\frac{C}{\eta^{3/2}}.
\end{align*}
For $n\geq4$,
\begin{align*}
    \left|\left\langle\prod_{m=1}^{n}G_{m}B_{m}\right\rangle\right|&\leq\left\langle\frac{Im G_{1}}{\eta_{1}}B_{1}\frac{\Im G_{2}}{\eta_{2}}B_{1}^{*}\right\rangle^{1/2}\\&\left\langle B_{2}\left(\prod_{m=3}^{n-2}G_{m}B_{m}\right)G_{n-1}B_{n-1}G_{n}B_{n}B_{n}^{*}G^{*}_{n}B_{n-1}^{*}G^{*}_{n-1}\left(B_{2}\prod_{m=3}^{n-2}G_{m}B_{m}\right)^{*}\right\rangle^{1/2}\\
    &\leq\frac{\|B_{n}\|\prod_{m=2}^{n-2}\|B_{m}\|}{\eta^{n-2}}\left\langle\Im\left(G_{1}\right)B_{1}\Im\left(G_{2}\right)B_{1}^{*}\right\rangle^{1/2}\left\langle\Im\left(G_{n-1}\right)B_{n-1}\Im\left(G_{n}\right)B_{n-1}^{*}\right\rangle^{1/2}\\
    &\leq\frac{C}{\eta^{n-2}}.
\end{align*}
In the second inequality we used $\tr XYX^{*}\leq \|X\|^{2}\tr Y$ for $Y>0$ to take $B_{2}\prod_{m=3}^{n-2}G_{m}B_{m}$ and $B_{n}B_{n}^{*}$ outside the trace.
\end{proof}

\section{Partial Schur Decomposition}
It is well known that a matrix of dimension $N$ can be brought into upper triangular form by a successively applying $N$ Householder transformations. If we stop after $k<N$ transformations then the resulting matrix will be partially upper triangular; this is the partial Schur decomposition. 
For $i=1,...,k$, define the manifolds
\begin{align}
    \Omega_{i}&:=\mbb{C}\times S^{N-i}_{+}\times\mbb{C}^{N-i},\\
    \Omega&:=\Omega_{1}\times\cdots\Omega_{k}\times \mbf{M}_{N-k}\left(\mbb{C}\right),
\end{align}
where $S^{n}_{+}=\{\mbf{v}\in S^{n}:v_{1}\geq0\}$. The volume elements are
\begin{align}
    d\omega_{i}\left(z_{i},\mbf{v}_{i},\mbf{w}_{i}\right)&:=dz_{i}dS_{N-i}\left(\mbf{v}_{i}\right)d\mbf{w}_{i},\\
    d\omega(z_{1},\mbf{v}_{1},\mbf{w}_{1},...,z_{k},\mbf{v}_{k},\mbf{w}_{k},M)&:=\left(\prod_{j=1}^{k}d\omega_{j}\left(z_{j},\mbf{v}_{j},\mbf{w}_{j}\right)\right)dM,
\end{align}
respectively, where $dS_{n}$ is the Haar measure on $S^{n}$ (with total measure equal to $\text{Vol}(S^{n}$)). Let $\pi:\Omega\to\mbb{C}^{k}$ denote the projection onto the $z_{i}$, i.e.
\begin{align}
    \pi\left(z_{1},\mbf{v}_{1},\mbf{w}_{1},...,z_{k},\mbf{v}_{k},\mbf{w}_{k},M\right)&=(z_{1},...,z_{k})=\mbf{z}.\label{eq:pi}
\end{align}
For $i=1,...,k$, define the maps
\begin{align}
    T_{i}&:\Omega_{i}\times  \mbf{M}_{N-i}\to  \mbf{M}_{N-i+1}\\
    T_{i}\left(z_{i},\mbf{v}_{i},\mbf{w}_{i},{M^{(i)}}\right)&=M^{(i-1)}=R\left(\mbf{v}_{i}\right)\begin{pmatrix}z_{i}&\mbf{w}_{i}^{*}\\0&{M^{(i)}}\end{pmatrix}R\left(\mbf{v}_{i}\right),
\end{align}
where $R\left(\mbf{v}_{i}\right)$ is the Householder transformation that exchanges $\mbf{v}_{i}$ with the first coordinate vector of {$\mbb{C}^{N-i+1}$} {(explicitly, $R(\mbf{v}_{i})=1_{N-i+1}-2\mbf{q}_{i}\mbf{q}_{i}^{*}$, where $\mbf{q}_{i}=\frac{\mbf{v}_{i}+\mbf{e}_{i}}{\|\mbf{v}_{i}+\mbf{e}_{i}\|}$ and $\mbf{e}_{i}$ is the first coordinate vector in $\mbb{C}^{N-i+1}$)}. Each $T_{i}$ corresponds to one step of the Schur decomposition. Denoting by $\text{Id}_{j}$ the identity on $\Omega_{1}\times\cdots\times\Omega_{j}$, the partial Schur decomposition is the map
\begin{align}
    T&:\Omega_{1}\times\cdots\Omega_{k}\times  \mbf{M}_{N-k}\to \mbf{M}_{N},\\
    T&=T_{1}\circ\left(\text{Id}_{1}\oplus T_{2}\right)\cdots\circ \left(\text{Id}_{k-1}\oplus T_{k}\right),
\end{align}
which takes the form
\begin{align}
    M&=T(z_{1},\mbf{v}_{1},\mbf{w}_{1},...,z_{k},\mbf{v}_{k},\mbf{w}_{k},M^{\left(k\right)})\nonumber\\
    &=U\begin{pmatrix}
    z_{1}&\mbf{w}_{1}^{*}V_{1}^{*}&&\\
    0&z_{2}&\mbf{w}_{2}^{*}V_{2}^{*}&\\
    \vdots&0&\ddots&\ddots\\
    0&\cdots&\ddots&z_{k}&\mbf{w}_{k}^{*}\\
    0&\cdots&0&0&M^{\left(k\right)}
    \end{pmatrix}U^{*},\label{eq:pSchur}
\end{align}
where
\begin{align}
    U&=\prod_{j=1}^{k}\begin{pmatrix}1_{j-1}&0\\0&R\left(\mbf{v}_{j}\right)\end{pmatrix}{\in\mbf{U}_{N}},\label{eq:U}\\
    V_{j}&=\prod_{l=0}^{k-j-1}\begin{pmatrix}1_{k-j-l-1}&0\\0&R\left(\mbf{v}_{k-l}\right)\end{pmatrix}{\in\mbf{U}_{N-j}}.
\end{align}

To understand its properties we need only understand those of the $T_{i}$, which we summarise in the following lemma:
\begin{lemma}
$T_{i}$ is surjective, and its Jacobian is given by
\begin{align}
    J_{i}\left(z_{i},\mbf{v}_{i},\mbf{w}_{i},{M^{(i)}}\right)&=|\textup{det}({M^{(i)}}-z_{i})|^{2}.
\end{align}
If {$M^{(i-1)}$} has distinct eigenvalues, then
{\begin{align}\label{eq:inverseimage_Ti}
    T_{i}^{-1}\left(M^{(i-1)}\right)&=\bigcup_{n=1}^{N-i+1}\left\{(z_{i,n},\mbf{v}_{i,n},\mbf{w}_{i,n},M^{\left(i\right)}_{n})\right\},
\end{align}}
for some $\mbf{v}_{i,n},\mbf{w}_{i,n},M^{\left(i\right)}_{n}$, where $\{z_{i,n}\}_{n=1}^{N-i+1}$ are the eigenvalues of $M^{(i-1)}$.
\end{lemma}
\begin{proof}
The surjectivity follows from the fact that any matrix has at least one eigenvector. The Jacobian has been calculated in \cite{fyodorov_absolute_2007}, eq. (6.5). If $M^{(i-1)}$ has distinct eigenvalues, then $\mbf{v}_{i,n}$ is uniquely determined by $z_{i,n}$. $\mbf{w}_{i,n}$ and $M^{\left(i\right)}_{n}$ are given by
\begin{align*}
    \begin{pmatrix}0\\\mbf{w}_{i,n}\end{pmatrix}&=R(\mbf{v}_{i,n})(1-\mbf{v}_{i,n}\mbf{v}^{*}_{i,n})M^{(i-1)*}\mbf{v}_{i,n}R(\mbf{v}_{i,n}),\\
    \begin{pmatrix}0&0\\0&M^{\left(i\right)}_{n}\end{pmatrix}&=R(\mbf{v}_{i,n})(1-\mbf{v}_{i,n}\mbf{v}^{*}_{i,n})M^{(i-1)}(1-\mbf{v}_{i,n}\mbf{v}^{*}_{i,n})R(\mbf{v}_{i,n}),
\end{align*}
and are therefore uniquely determined by $\mbf{v}_{i,n}$.
\end{proof}

\begin{cor}\label{cor1}
The Jacobian of $T$ is
\begin{align}
    J&=|\Delta\left(\mbf{z}\right)|^{2}\prod_{i=1}^{k}\left|\textup{det}(M^{\left(k\right)}-z_{i})\right|^{2},
\end{align}
{where $\Delta(\mbf{z})=\prod_{i<j}(z_{i}-z_{j})$ is the Vandermonde determinant,} and for matrices $M\in\mbf{M}_{N}$ with distinct eigenvalues $z_{i},\,i=1,...,N$, we have
{\begin{align}\label{eq:inverseimage_T}
    \pi(T^{-1}\left(M\right))&=\bigcup_{i_{1}\neq\cdots\neq i_{k}}\left\{\left(z_{i_{1}},...,z_{i_{k}}\right)\right\}.
\end{align}}
\end{cor}
\begin{proof}
Note that for $i=2,...,k$ we have
\begin{align*}
    \det\left(M^{(i)}-z_{i}\right)&=\left({z_{i+1}}-z_{i}\right)\det\left(M^{(i+1)}-z_{i}\right),
\end{align*}
so that
\begin{align*}
    J&=\prod_{i=1}^{k}J_{i}\\
    &=|\Delta\left(\mbf{z}\right)|^{2}\prod_{i=1}^{k}\left|\det\left(M^{\left(k\right)}-z_{i}\right)\right|^{2}.
\end{align*}
\eqref{eq:inverseimage_T} follows directly from \eqref{eq:inverseimage_Ti} and $T^{-1}=\left(\text{Id}_{k-1}\oplus T_{k}\right)^{-1}\circ\cdots\circ\left(\text{Id}_{1}\oplus T_{2}\right)^{-1}\circ T_{1}^{-1}$.
\end{proof}

Now consider a random matrix $M$ with density $\rho$ with respect to the Lebesgue measure.
\begin{prop}
For any bounded and measurable $f:\mbb{C}^{k}\to\mbb{R}$ we have
\begin{align}\label{eq:fexp}
    \mbb{E}\left[\sum_{{i_{1}\neq\cdots\neq i_{k}}}f\left(z_{i_{1}},...,z_{i_{k}}\right)\right]&=\int_{\Omega}f(\pi\left(x\right))\rho({T}\left(x\right))J(x)d\omega(x)
\end{align}
\end{prop}
\begin{proof}
Since matrices with density have distinct eigenvalues with probability one, we have by \Cref{cor1}
\begin{align}
    \mbb{E}\left[\sum_{{i_{1}\neq\cdots\neq i_{k}}}f\left(z_{i_{1}},...,z_{i_{k}}\right)\right]&=\mbb{E}\left[\sum_{x\in T^{-1}\left(M\right)}f(\pi\left(x\right))\right].
\end{align}
Recall the area formula (see e.g. Section 3.2 in the book \cite{federer_geometric_1996}): for $m\geq n$, a Lipschitz function $T:\mbb{R}^{n}\to\mbb{R}^{m}$ with Jacobian $J$ and a $g:\mbb{R}^{n}\to\mbb{R}$ such that $g\geq0$ or $gJ$ is integrable, we have
\begin{align*}
    \int_{\mbb{R}^{n}}g\left(x\right)J(x)dx&=\int_{\mbb{R}^{m}}\sum_{x\in T^{-1}\left(y\right)}g(x)d\mc{H}^{n}(y),
\end{align*}
where $\mc{H}^{n}$ is the $n$-dimensional Hausdorff measure (when $n=m$ this is simply the Lebesgue measure). This readily extends to the case of a Lipschitz function $T:\mathfrak{M}\to\mbb{R}^{m}$, where $\mathfrak{M}$ is an $n$-dimensional Riemannian manifold and we replace $dx$ on the left hand side above by the volume measure on $\mathfrak{M}$. Now take $n=m=2N^{2}$, $\mathfrak{M}=\Omega$, $T$ to be the partial Schur decomposition  and $g=\rho\circ T$. Let $K_{j}$ be an increasing sequence of compact sets whose union is $ \mbf{M}_{N}$ and set $N_{j}=T^{-1}\left(K_{j}\right)$. Then $N_{j}$ is compact and the restriction of $T$ to $N_{j}$ is Lipschitz. Thus we can apply the area formula to obtain
\begin{align*}
    \int_{N_{j}}\rho(T\left(x\right))J(x)d\omega(x)&=\int_{K_{j}}\sum_{x\in T^{-1}\left(M\right)}\rho(T\left(x\right))dM=\begin{pmatrix}N\\k\end{pmatrix}\int_{K_{j}}\rho\left(M\right)dM.
\end{align*}
Taking the limit $j\to\infty$ by the monotone convergence theorem we conclude that $\left(\rho\circ T\right)J\in L^{1}\left(\Omega\right)$. Applying the area formula again with $g=\left(f\circ \pi\right)(\rho\circ T)$ we obtain \eqref{eq:fexp}.
\end{proof}

From this follows the main result of this section, namely an integral formula for the $k$-point function $\rho^{\left(k\right)}_{M}$ of random matrices $M$ with continuous distributions.
\begin{cor}\label{cor2}
{Let $M$ be a random matrix with density $\rho$ with respect to the Lebesgue measure.} The $k$-point correlation function satisfies
\begin{align}\label{eq:kpointint}
    \rho^{\left(k\right)}_{M}\left(\mbf{z}\right)&=\frac{|\Delta\left(\mbf{z}\right)|^{2}}{\left(2\pi\right)^{k}}\int_{S^{N-1}}\int_{\mbb{C}^{N-1}}\cdots\int_{S^{N-k}}\int_{\mbb{C}^{N-k}}\int_{ \mbf{M}_{N-k}}\rho(T({z_{1}},\mbf{v}_{1},\mbf{w}_{1},...,{z_{k}},\mbf{v}_{k},\mbf{w}_{k},M^{\left(k\right)}))\nonumber\\
    &\times\left(\prod_{i=1}^{k}\left|\textup{det}(M^{\left(k\right)}-z_{i})\right|^{2}\right)dM^{\left(k\right)}dS_{N-k}\left(\mbf{v}_{k}\right)d\mbf{w}_{k}\cdots dS_{N-1}\left(\mbf{v}_{1}\right)d\mbf{w}_{1}.
\end{align}
\end{cor}
\begin{proof}
The $k$-point function is defined by
\begin{align*}
    \mbb{E}\left[\sum_{{i_{1}\neq\cdots\neq i_{k}}}f\left(z_{i_{1}},...,z_{i_{k}}\right)\right]&=\int_{\mbb{C}^{k}}f\left(\mbf{z}\right)\rho^{\left(k\right)}_{N}\left(\mbf{z}\right)d\mbf{z}.
\end{align*}
Equating this with the right hand side of \eqref{eq:fexp} we can read off \eqref{eq:kpointint}. We have lifted the integrals over $S^{N-j}_{+}$ to $S^{N-j}$ using the invariance of the integrand and measure under transformations $\mbf{v}_{j}\mapsto e^{i\phi}\mbf{v}_{j}$. Hence the additional factor of $1/\left(2\pi\right)^{k}$ in \eqref{eq:kpointint}.
\end{proof}

\section{Proof of \Cref{thm1}}
In the following we fix $\epsilon>0$, $z\in\mbb{C}$ such that $|z|<1$, and $t\geq N^{-1/3+\epsilon}$. Let $A$ be a deterministic matrix satisfying \eqref{eq:A_g}, \eqref{eq:A_alpha}, \eqref{eq:A_beta}, \eqref{eq:A_gamma} and \eqref{eq:A2}, and $B$ be a complex Ginibre matrix. Define the Gauss-divisible matrix $M_{t}=A+\sqrt{t}B$; this matrix has the following density with respect to the Lebesgue measure on $\mbf{M}_{N}$:
\begin{align*}
    \rho\left(M_{t}\right)&=\left(\frac{N}{\pi t}\right)^{N^{2}}e^{-\frac{N}{t}\tr\left(M_{t}-A\right)^{*}\left(M_{t}-A\right)}.
\end{align*}

Our goal is to evaluate the asymptotics of the integral formula in \eqref{eq:kpointint} for the $k$-point function. For this we need an expression for $\rho\left(M_{t}\right)$ in terms of the new variables of the partial Schur decomposition. Take $M_{t}$ as in the right hand side of \eqref{eq:pSchur} with $U$ as in \eqref{eq:U}. Then we have
\begin{align}
    U^{*}AU&=\begin{pmatrix}a_{1}&\mbf{b}_{1}^{*}V_{1}^{*}&&&\\
    V_{1}\mbf{c}_{1}&a_{2}&\mbf{b}_{2}^{*}V_{2}^{*}&&\\
    &V_{2}\mbf{c}_{2}&\ddots&\ddots&\\
    &&\ddots&a_{k}&\mbf{b}_{k}^{*}\\
    &&&\mbf{c}_{k}&A^{\left(k\right)}\end{pmatrix},\label{eq:UAU}
\end{align}
{where the variables $a_{i},\mbf{b}_{i},\mbf{c}_{i}$ and $A^{(i)}$ are defined recursively by
\begin{align}
    R(\mbf{v}_{i})A^{(i-1)}R(\mbf{v}_{i})&:=\begin{pmatrix}a_{i}&\mbf{b}_{i}^{*}\\\mbf{c}_{i}&A^{(i)}\end{pmatrix},
\end{align}
with $A^{\left(0\right)}=A$, and thus depend on $\mbf{v}_{j},\,j\leq i$. Explicitly,}
\begin{align}
    a_{i}&=\mbf{v}_{i}^{*}A^{\left(i-1\right)}\mbf{v}_{i},\nonumber\\
    \begin{pmatrix}0\\\mbf{b}_{i}\end{pmatrix}&=R(\mbf{v}_{i})\left(1-\mbf{v}_{i}\mbf{v}_{i}^{*}\right)A^{(i-1)*}\mbf{v}_{i},\nonumber\\
    \begin{pmatrix}0\\\mbf{c}_{i}\end{pmatrix}&=R(\mbf{v}_{i})\left(1-\mbf{v}_{i}\mbf{v}_{i}^{*}\right)A^{(i-1)}\mbf{v}_{i},\nonumber\\
    \begin{pmatrix}0&0\\0&A^{\left(i\right)}\end{pmatrix}&=R(\mbf{v}_{i})\left(1-\mbf{v}_{i}\mbf{v}_{i}^{*}\right)A^{(i-1)}\left(1-\mbf{v}_{i}\mbf{v}_{i}^{*}\right)R(\mbf{v}_{i}).\label{eq:Arecurrence}
\end{align}
The matrices $A^{\left(i\right)}$ are projections of $A$ onto subspaces of codimension $i$. We will denote by a superscript $\left(i\right)$ any quantity for which we replace $A$ by $A^{\left(i\right)}$ and $1_{N}$ by $1_{N-i}$ in its definition, e.g. $\mc{H}^{\left(i\right)}_{z}=\begin{pmatrix}0&A^{\left(i\right)}-z\\A^{\left(i\right)*}-\bar{z}&0\end{pmatrix}$, $G^{\left(i\right)}_{z}\left(\eta\right)=(\mc{H}^{\left(i\right)}_{z}-i\eta)^{-1}$, $E^{(i)}=\begin{pmatrix}0&1\\0&0\end{pmatrix}\otimes1_{N-i}$ and so on.

In terms of the variables on the right hand side of \eqref{eq:UAU} and $M^{\left(k\right)}_{t}\in \mbf{M}_{N-k}$, we have
\begin{align*}
    \rho\left(M_{t}\right)&=\left(\frac{N}{\pi t}\right)^{N^{2}}\exp\left\{-\frac{N}{t}\sum_{i=1}^{k}\left[|z_{i}-a_{i}|^{2}+\left\|\mbf{w}_{i}-\mbf{b}_{i}\right\|^{2}+\left\|\mbf{c}_{i}\right\|^{2}\right]\right.\nonumber\\
    &\left.-\frac{N}{t}\tr(M^{\left(k\right)}_{t}-A^{\left(k\right)})^{*}(M^{\left(k\right)}_{t}-A^{\left(k\right)})\right\}.
\end{align*}
We need to insert the above expression into \eqref{eq:kpointint} and integrate over $\mbf{v}_{i},\,\mbf{w}_{i},\,i=1,...,k$ and $M^{\left(k\right)}_{t}$.

The integrand depends on $\mbf{w}_{i}$ only through the term $e^{-\frac{N}{t}\left\|\mbf{w}_{i}-\mbf{b}_{i}\right\|^{2}}$, and so the integral over $\mbf{w}_{i}$ gives a constant
\begin{align*}
    \int_{\mbb{C}^{N-i}}e^{-\frac{N}{t}\left\|\mbf{w}_{i}-\mbf{b}_{i}\right\|^{2}}d\mbf{w}_{i}&=\left(\frac{\pi t}{N}\right)^{N-i}.
\end{align*}
To simplify the notation for the integral formula further, we define the following probability measures on the spheres $S^{N-i}$:
\begin{align}
    d\nu_{i}\left(\mbf{v}_{i}\right)&:=\frac{1}{K_{i}(z_{i};A^{\left(i-1\right)})}e^{-\frac{N}{t}\left\|(A^{\left(i-1\right)}-z_{i})\mbf{v}_{i}\right\|^{2}}dS_{N-i}\left(\mbf{v}_{i}\right).
\end{align}
where
\begin{align}
    K_{i}(z_{i};A^{\left(i-1\right)})&=\int_{S^{N-i}}e^{-\frac{N}{t}\left\|(A^{\left(i-1\right)}-z_{i})\mbf{v}_{i}\right\|^{2}}dS_{N-i}\left(\mbf{v}_{i}\right).
\end{align}
Note that $\nu_{i}$ and $K_{i}$ depend on $\mbf{v}_{l},\,l=1,...,i-1$ through $A^{\left(i-1\right)}$ but we have suppressed this dependence to simplify notation. We should therefore keep in mind that we have to integrate over $\mbf{v}_{i}$ in descending order of $i$.

We now have the following formula for the $k$-point function:
\begin{align}
    \rho^{\left(k\right)}_{N}\left(\mbf{z};A\right)&=\int_{S^{N-1}}\cdots\int_{S^{N-k}}F(\mbf{z};A^{\left(k\right)})\left(\prod_{i=1}^{k}b_{N,i}K_{i}(z_{i};A^{\left(i-1\right)})\right)d\nu_{k}\left(\mbf{v}_{k}\right)\cdots d\nu_{1}\left(\mbf{v}_{1}\right),
\end{align}
where
\begin{align}
    F(\mbf{z};A^{\left(k\right)})&:=d_{N}|\Delta\left(\mbf{z}\right)|^{2}\int_{ \mbf{M}_{N-k}}e^{-{N}\tr B^{\left(k\right)*}B^{\left(k\right)}}\prod_{i=1}^{k}\left|\det\left(A^{\left(k\right)}+{\sqrt{t}}B^{\left(k\right)}-z_{i}\right)\right|^{2}dB^{\left(k\right)},
\end{align}
and
\begin{align}
    b_{N,i}&=\frac{1}{t}\sqrt{\frac{N}{2\pi t}}\left(\frac{N}{\pi t}\right)^{N-i},\\
    d_{N}&=\frac{1}{\pi^{k}}\left(\frac{Nt}{2\pi}\right)^{k/2}\left(\frac{N}{\pi t}\right)^{\left(N-k\right)^{2}}.
\end{align}

We have split the constant $\left(N/\pi t\right)^{N^{2}}$ into $b_{N,i}$ and $d_{N}$ to reflect the scaling of $K_{i}$ and $F$. Note that $F$ can be rewritten as an expectation value with respect to a Gauss-divisible matrix of dimension $N-k$, i.e. we can write
\begin{align*}
    F(\mbf{z};A^{\left(k\right)})&=d_{N}|\Delta\left(\mbf{z}\right)|^{2}\mbb{E}\left[\prod_{i=1}^{k}\left|\det\left(M^{\left(k\right)}_{t}-z_{i}\right)\right|^{2}\right],
\end{align*}
where
\begin{align*}
    M^{\left(k\right)}_{t}&=A^{\left(k\right)}+\sqrt{t}B^{\left(k\right)}
\end{align*}
and the expectation is taken with respect to the complex Gaussian matrix $B^{\left(k\right)}$ whose entries have mean zero and variance $1/N$.

Let us now prove the first statement of \Cref{thm1}, namely that there is a unique $\eta_{z,t}$ such that
\begin{align}\label{eq:trG}
    \frac{t}{2\eta_{z,t}}\Im\left\langle G_{z}\left(\eta_{z,t}\right)\right\rangle&=t\left\langle H_{z}\left(\eta_{z,t}\right)\right\rangle=1,
\end{align}
and
\begin{align}
    c_{1}t\leq\eta_{z,t}\leq c_{2}t.\label{eq:etaineq}
\end{align}
Let $\Omega=[c_{1}t,c_{2}t]$, where $c_{1},c_{2}$ depend on the constants in \eqref{eq:A_g}, and $f\left(\eta\right)=\frac{t}{2}\Im\left\langle G_{z}\left(\eta\right)\right\rangle=t \Im g_{z}\left(\eta\right)$. Any non-zero solution of \eqref{eq:trG} is a fixed point of $f$. From the first assumption \eqref{eq:A_g}, we have
\begin{align*}
    f\left(\Omega\right)\subset\Omega;
\end{align*}
since $f$ is continuous it must therefore have at least one fixed point in $\Omega$. On the other hand, $\left\langle H_{z}\left(\eta\right)\right\rangle$ is monotonic in $\eta$ and so \eqref{eq:trG} has at most one solution.

Let $\alpha_{z,t}:=\alpha_{z}\left(\eta_{z,t}\right)$ and similarly for $\beta_{z,t},\gamma_{z,t}$, where these quantities were defined in \eqref{eq:alpha}, \eqref{eq:beta} and \eqref{eq:gamma}. Henceforth we will suppress the argument of the resolvents $G_{z},H_{z}$ and $\wt{H}_{z}$ when they are evaluated at $\eta_{z,t}$, i.e. $G_{z}:=G_{z}\left(\eta_{z,t}\right)$. {Fix $z\in\mbb{C}$, $\mbf{z}\in\mbb{C}^{k}$ and define the rescaled parameters
\begin{align}
    \mbf{w}&:=z+\frac{\mbf{z}}{\sqrt{N\sigma_{z,t}}}.\label{eq:w}
\end{align}
We are interested in the shifted and rescaled correlation function}
\begin{align*}
    \wt{\rho}^{\left(k\right)}_{N}\left(\mbf{z};A\right)&:=\frac{1}{(N\sigma_{z,t})^{k}}\rho^{\left(k\right)}_{N}\left(\mbf{w};A\right)\\
    &=\int_{S^{N-1}}\cdots\int_{S^{N-k}}\wt{F}(\mbf{z};A^{\left(k\right)})\left(\prod_{i=1}^{k}\wt{K}_{i}(z_{i};A^{\left(i-1\right)})\right)d\nu_{k}\left(\mbf{v}_{k}\right)\cdots d\nu_{1}\left(\mbf{v}_{1}\right),
\end{align*}
where
\begin{align}
    \wt{K}_{i}\left(z_{i}\right)&:=b_{N,i}K_{i}(w_{i};A^{\left(i-1\right)}),\label{eq:Ktilde}\\
    \wt{F}(\mbf{z};A^{\left(k\right)})&:=\frac{1}{(N\sigma_{z,t})^{k}}F(\mbf{w};A^{\left(k\right)}).\label{eq:Ftilde}
\end{align}
The second statement of the theorem is equivalent to the claim that
\begin{align}
    \lim_{N\to\infty}\wt{\rho}^{\left(k\right)}_{N}\left(\mbf{z};A\right)&=\rho^{\left(k\right)}_{GinUE}\left(\mbf{z}\right),
\end{align}
uniformly on compact subsets of $\mbb{C}^{k}$.

Define $V_{i}=1_{2}\otimes\mbf{v}_{i}{\in\mbb{C}^{2(N-i+1)\times2}}$ and
\begin{align}
    \phi_{z,t}&:=\frac{\eta_{z,t}^{2}}{t}-\frac{1}{N}\log\left|\det(\mc{H}_{z}-i\eta_{z,t})\right|,\label{eq:phi}\\
    \psi_{i}&:=\exp\left\{-\sqrt{\frac{N}{\sigma_{z,t}}}\left\langle G_{z}Z_{i}\right\rangle-\Re\left(\bar{\tau}_{z,t}z_{i}^{2}\right)+|z_{i}|^{2}\right\},\label{eq:psi}
\end{align}
where $Z_{i}=\begin{pmatrix}0&z_{i}\\\bar{z}_{i}&0\end{pmatrix}$ and
\begin{align}
    \tau_{z,t}&:=\frac{\beta_{z,t}^{2}+\gamma_{z,t}\delta_{z,t}}{\gamma_{z,t}\sigma_{z,t}},\label{eq:tau}\\
    \delta_{z,t}&:=\left\langle\left(H_{z}(A-z)\right)^{2}\right\rangle.\label{eq:delta}
\end{align}
The following lemma gives the asymptotic behaviour of $F$ and $K_{i}$.
\begin{lemma}
We have
\begin{align}
    \wt{F}\left(\mbf{z};A^{\left(k\right)}\right)&=\left[1+O\left(\frac{1}{\sqrt{Nt^{3}}}\right)\right]\left(\frac{t^{3}\gamma_{z,t}}{\eta_{z,t}^{2}}\right)^{k/2}\rho^{\left(k\right)}_{GinUE}\left(\mbf{z}\right)\nonumber\\&\prod_{i=1}^{k}\left(\frac{\eta^{2}_{z,t}\left|\det V_{i}^{*}G^{\left(i-1\right)}_{z}V_{i}\right|}{t^{2}\gamma_{z,t}\sigma_{z,t}}\right)^{k}\psi_{i}e^{-N\phi_{z,t}},\label{eq:Fasymptotic}
\end{align}
and for $i=1,...,k$
\begin{align}
    \wt{K}_{i}\left(z_{i};A^{\left(i-1\right)}\right)&=\left[1+O\left(\frac{\log N}{\sqrt{Nt^{2}}}\right)\right]\sqrt{\frac{\eta^{2}_{z,t}}{t^{3}\gamma_{z,t}}}\frac{e^{N\phi_{z,t}}}{\psi_{i}}\prod_{j=1}^{i-1}\left|\det^{-1}V_{j}^{*}G^{(j-1)}_{z}V_{j}\right|,\label{eq:Kasymptotic}
\end{align}
where the error terms are uniform in $\mbf{v}_{j}$, $j=1,...,k$. If \eqref{eq:A3} holds, then the error becomes $O\left(\frac{\log N}{\sqrt{Nt^{2}}}\right)$.
\end{lemma}
{The appearance of $A^{(j)},\,j=1,...,k-1$ in the right-hand sides above can be anticipated by observing that Cramer's rule implies that
\begin{align*}
    \frac{\det(\mc{H}^{(j+1)}_{z}-i\eta)}{\det(\mc{H}^{(j)}_{z}-i\eta)}&=\det V_{j}^{*}G^{(j)}_{z}(\eta)V_{j}.
\end{align*}}
We note that a result analogous to \eqref{eq:Fasymptotic} for the asymptotics of the expectation of a product of characteristic polynomials of a Ginibre matrix perturbed by a fixed matrix of finite rank was obtained by Liu and Zhang in Theorem 3.1 of \cite{liu_phase_2022}

Replacing $F$ and $\wt{K}_{i}$ by their asymptotic approximations, we see that the correlation function reduces to
\begin{align*}
    \wt{\rho}^{\left(k\right)}_{N}\left(\mbf{z};A\right)&=\left[1+O\left(\frac{1}{\sqrt{Nt^{3}}}\right)\right]\rho^{\left(k\right)}_{GinUE}\left(\mbf{z}\right)\\
    &\int_{S^{N-1}\times\cdots S^{N-k}}\left[\prod_{i=1}^{k}\left(\frac{\eta^{2}_{z,t}|\det V_{i}^{*}G^{\left(i-1\right)}_{z}V_{i}|}{t^{2}\gamma_{z,t}\sigma_{z,t}}\right)^{i}\right]d\nu_{k}(\mbf{v}_{k})\cdots d\nu_{1}(\mbf{v}_{1}).
\end{align*}
Therefore the theorem will follow immediately from
\begin{lemma}\label{lem:nuExpectation}
For $i=1,...,k$ we have
\begin{align}
    \left(\frac{\eta^{2}_{z,t}}{t^{2}\gamma_{z,t}\sigma_{z,t}}\right)^{i}\int_{S^{N-i}}|\det V_{i}^{*}G^{\left(i-1\right)}_{z}V_{i}|^{i}d\nu_{i}\left(\mbf{v}_{i}\right)&=1+O\left(\frac{1}{\sqrt{Nt^{3}}}\right),\label{eq:qformexpectation}
\end{align}
uniformly in $\mbf{v}_{j},\,j=1,...,i-1$. If \eqref{eq:A3} holds, then the error becomes $O\left(\frac{\log N}{\sqrt{Nt^{2}}}\right)$.
\end{lemma}
We will prove \eqref{eq:Fasymptotic} in \Cref{sec:F} and \eqref{eq:Kasymptotic} and \eqref{eq:qformexpectation} in \Cref{sec:K}.

\section{Expectation of Characteristic Polynomials}\label{sec:F}
Using Berezin integration, one can reduce the expectation of a product of $k$ characteristic polynomials with respect to Gaussian matrices to an integral over $k\times k$ complex matrices. This is an example of so called duality formulas; this particular case can be found in eq. (39) of \cite{grela_diffusion_2016} and Theorem 2.4 of \cite{liu_phase_2022}. A special case when $A^{(k)}=0$ and the $z_{j}$ take two distinct values is given in eq. (24) of \cite{nishigaki_replica_2002}. {A general reference for Grassmann integration in random matrix theory and beyond is Wegner's book \cite{wegner_supermathematics_2016}.}
\begin{lemma}
Let $\wt{F}$ be as defined in \eqref{eq:Ftilde}. Then
\begin{align}
    \wt{F}\left(\mbf{z};A^{\left(k\right)}\right)&=f_{N,k}|\Delta\left(\mbf{z}\right)|^{2}\int_{\mbf{M}_{k}}e^{-\frac{N}{t}\tr X^{*}X}\det\left[i M^{\left(k\right)}(X)\right] dX,\label{eq:Ftilde2}
\end{align}
where
\begin{align}
     M^{\left(k\right)}(X)&=\begin{pmatrix}-iX\otimes1_{N-k}&1_{k}\otimes A^{\left(k\right)}-\mbf{w}\otimes1_{N-k}\\1_{k}\otimes A^{\left(k\right)*}-\bar{\mbf{w}}\otimes1_{N-k}&-iX^{*}\otimes1_{N-k}\end{pmatrix},
\end{align}
and
\begin{align}
    f_{N,k}&=\frac{N^{k^{2}/2}}{2^{k/2}\pi^{k\left(k+3/2\right)}t^{k\left(k-1/2\right)}\sigma^{k(k+1)/2}_{z,t}}.\label{eq:fNk}
\end{align}
\end{lemma}
\begin{proof}
{Let $\chi_{j,l},\,\chi_{j,l}^{*},\,\psi_{j,l},\,\psi_{j,l}^{*}$, $j=1,...,k,\,l=1,...,N-k$ be mutually anti-commuting variables, i.e. satisfy the relations
\begin{align*}
    \chi_{j,l}\chi^{*}_{k,m}&=-\chi^{*}_{k,m}\chi_{j,l},\quad\chi_{j,l}\psi_{k,m}=-\psi_{k,m}\chi_{j,l},\quad\psi_{j,l}\psi^{*}_{k,m}=-\psi^{*}_{k,m}\psi_{j,l}.
\end{align*}
The superscript $*$ is merely notational: $\chi^{*}_{j,l}$ and $\chi_{j,l}$ are two independent anti-commuting variables. Let $\chi_{j}=(\chi_{j,1},...,\chi_{j,N-k})^{T},\,\psi_{j}=(\psi_{j,1},...,\psi_{j,N-k})^{T},\,j=1,...,k$ be vectors constructed from these variables}; then we have
\begin{align*}
    \left|\det\left(A^{\left(k\right)}+{\sqrt{t}}B^{(k)}-w_{j}\right)\right|^{2}&=\int e^{-\chi_{j}^{*}(A^{\left(k\right)}+{\sqrt{t}}B^{(k)}-w_{j})\chi_{j}-\psi_{j}^{*}(A^{\left(k\right)*}+{\sqrt{t}}B^{(k)*}-\bar{w}_{j})\psi_{j}}d\chi d\psi.
\end{align*}
The integral over $B^{(k)}$ is now a Gaussian integral:
\begin{align*}
    &\left(\frac{N}{\pi t}\right)^{\left(N-k\right)^{2}}\int_{ \mbf{M}_{N-k}}e^{-{N}\tr B^{(k)*}B^{(k)}-{\sqrt{t}}\tr B^{(k)}\sum_{j=1}^{k}\chi_{j}\chi_{j}^{*}-{\sqrt{t}}\tr B^{(k)*}\sum_{j=1}^{k}\psi_{j}\psi_{j}^{*}}dB^{(k)}\\&=e^{-\frac{t}{N}\sum_{j,l=1}^{k}\left(\chi_{j}^{*}\psi_{l}\right)(\psi_{l}^{*}\chi_{j})}
\end{align*}
The quartic term can be rewritten in terms of quadratic forms by introducing an auxilliary integral over $\mbf{M}_{k}$:
\begin{align*}
    e^{-\frac{t}{N}\sum_{j,l=1}^{k}\left(\chi_{j}^{*}\psi_{l}\right)(\psi_{l}^{*}\chi_{j})}&=\left(\frac{N}{\pi t}\right)^{k^{2}}\int_{\mbf{M}_{k}}e^{-\frac{N}{t}\tr X^{*}X-i\sum_{j,l=1}^{k}\left(X_{jl}\chi_{j}^{*}\psi_{l}+\bar{X}_{jl}\psi_{l}^{*}\chi_{j}\right)}dX.
\end{align*}
Now that we have only quadratic forms in $\chi_{j}$ and $\psi_{j}$, we can evaluate the integral over these vectors:
\begin{align*}
    &\int \exp\left\{-\begin{pmatrix}\chi^{*}&\psi^{*}\end{pmatrix}\begin{pmatrix}\mbf{w}\otimes1_{N-k}-1_{k}\otimes A^{\left(k\right)}&i X\otimes1_{N-k}\\iX^{*}\otimes1_{N-k}&\bar{\mbf{w}}\otimes1_{N-k}-1_{k}\otimes A^{\left(k\right)*}\end{pmatrix}\begin{pmatrix}\chi&\psi\end{pmatrix}\right\}d\chi d\psi\\
    &=\left(-1\right)^{N}\det\begin{pmatrix}\mbf{w}\otimes1_{N-k}-1_{k}\otimes A^{\left(k\right)}&i X\otimes1_{N-k}\\iX^{*}\otimes1_{N-k}&\bar{\mbf{w}}\otimes1_{N-k}-1_{k}\otimes A^{\left(k\right)*}\end{pmatrix}\\
    &=\det\left[i M^{\left(k\right)}(X)\right].
\end{align*}
\end{proof}

The benefit of this is that $N$ now appears as a parameter in the integrand and we can apply Laplace's method to obtain the asymptotics of $\wt{F}$. {Define the domain
\begin{align}
    \Omega&:=\left\{|s_{j}(X)-\eta_{z,t}|<\sqrt{\frac{t}{N}}\log N,\,j=1,...,k\right\},\label{eq:targetDomain}
\end{align}
where $s_{j}(X)$ are the singular values of $X$. The first step is to restrict the integral in \eqref{eq:Ftilde2} to $\Omega$.}
{\begin{lemma}
There is an absolute constant $c>0$ such that
\begin{align}
    \left|f_{N,k}|\Delta(\mbf{z})|^{2}\int_{\mbf{M}_{k}\setminus\Omega}e^{-\frac{N}{t}\tr X^{*}X}\det\left[iM^{(k)}(X)\right]dX\right|&\leq e^{-c\log^{2}N}
\end{align}
for sufficiently large $N$.
\end{lemma}}
\begin{proof}
Let $P=XX^{*}$ and $Q=X^{*}X$. {Instead of working directly with $\Omega$, we will consider a smaller domain
\begin{align}
    \Omega'&:=\Omega_{1}\cap\Omega_{2},
\end{align}
where
\begin{align}
    \Omega_{1}&:=\left\{|P_{ii}-\eta_{z,t}|<\sqrt{\frac{t^{3}}{N}}\log N,\,|Q_{ii}-\eta_{z,t}|<\sqrt{\frac{t^{3}}{N}}\log N,\,i=1,...,k\right\},\label{eq:Omega1}
\end{align}
and
\begin{align}
    \Omega_{2}&:=\left\{|P_{ij}|<\sqrt{\frac{t^{3}}{N}}\log N,\,1\leq i\neq j\leq k\right\}.\label{eq:Omega2}
\end{align}
The fact that $\Omega'\subset\Omega$ follows by Gershgorin's theorem (strictly speaking, for this inclusion we should modify the bound in the definition of $\Omega$ to be $C\sqrt{\frac{t}{N}}\log N$ for some absolute constant $C$, but this is irrelevant). To prove the lemma it therefore suffices to show that
\begin{align}
    f_{N,k}|\Delta(\mbf{z})|^{2}\int_{\mbf{M}_{k}\setminus\Omega'}e^{-\frac{N}{t}\tr X^{*}X}\left|\det\left[iM^{(k)}(X)\right]\right|dX&\leq e^{-c\log^{2}N}.\label{eq:MminusOmega'}
\end{align}
To prove this we will first show that the integral over $\mbf{M}_{k}\setminus\Omega_{1}$ is small, and then show that the integral over $\Omega_{1}\setminus\Omega_{2}$ is small.}

A first application of Fischer's inequality to $\det M^{(k)}(X)$ gives
\begin{align}
    |\det M^{\left(k\right)}(X)|&\leq\det^{1/2}\left(P\otimes1+|1\otimes A^{\left(k\right)*}-\bar{\mbf{w}}\otimes1|^{2}\right)\nonumber\\
    &\times\det^{1/2}\left(Q\otimes1+|1\otimes A^{\left(k\right)}-\mbf{w}\otimes1|^{2}\right).\label{eq:fischer1}
\end{align}
Applying Fischer's inequality again to each factor on the right hand above side we find
\begin{align}
    |\det M^{\left(k\right)}(X)|&\leq\prod_{i=1}^{k}\det^{1/2}(P_{ii}+|A^{\left(k\right)}-w_{i}|^{2})\det^{1/2}(Q_{ii}+|A^{\left(k\right)}-w_{i}|^{2}).\label{eq:fischer2}
\end{align}
Thus we have the following bound for the integrand in \eqref{eq:Ftilde2}:
\begin{align*}
    \prod_{i=1}^{k}e^{-\frac{N}{2t}(P_{ii}+Q_{ii})}\det^{1/2}\left(P_{ii}+|A^{(k)}-w_{i}|^{2}\right)\det^{1/2}\left(Q_{ii}+|A^{(k)}-w_{i}|^{2}\right).
\end{align*}
Now we multiply the above expression by $e^{Nk\phi_{z,t}}$, where $\phi_{z,t}$ was defined in \eqref{eq:phi}. Note that
\begin{align*}
    N\phi_{z,t}&=\eta_{z,t}^{2}\tr H_{z}-\log\det(\eta_{z,t}^{2}+|A-z|^{2}),
\end{align*}
so we are dividing by $\det^{k}(\eta_{z,t}^{2}+|A-z|^{2})$ and multiplying by $e^{k\eta_{z,t}^{2}\tr H_{z}}$. If we take out a factor of $\det\left(\eta^{2}_{z,t}+|A^{\left(k\right)}-w_{i}|^{2}\right)$ from each determinant in \eqref{eq:fischer2} we obtain
\begin{align*}
    \frac{\det\left(P_{ii}+|A^{(k)}-w_{i}|^{2}\right)}{\det\left(\eta_{z,t}^{2}+|A^{(k)}-w_{i}|^{2}\right)}&=\det\left(1+(P_{ii}-\eta_{z,t}^{2})H^{(k)}_{w_{i}}\right)\\
    &\leq\exp\left\{(P_{ii}-\eta_{z,t}^{2})\tr H^{(k)}_{w_{i}}-\frac{3(P_{ii}-\eta_{z,t}^{2}){^{2}}}{6+4|P_{ii}-\eta_{z,t}^{2}|/\eta_{z,t}^{2}}\tr\left((H^{(k)}_{w_{i}})^{2}\right)\right\},
\end{align*}
where we have used the inequality $\log(1+x)\leq x-\frac{3x^{2}}{6+4x}$. {Since $\Tr{H_{z}}=1/t$ by the definition of $\eta_{z,t}$, we have
\begin{align*}
    \exp\left\{-\frac{N}{2t}P_{ii}+\frac{N}{2}\eta^{2}_{z,t}\Tr{H_{z}}\right\}&=\exp\left\{-\frac{N}{2}(P_{ii}-\eta^{2}_{z,t})\Tr{H_{z}}\right\}.
\end{align*}}
This gives us the bound
\begin{align}
    e^{Nk\phi_{z,t}}|\wt{F}|&\leq |\Delta(\mbf{z})|^{2}\left(\prod_{i=1}^{k}\frac{\det\left(\eta^{2}_{z,t}+|A^{\left(k\right)}-w_{i}|^{2}\right)}{\det\left(\eta^{2}_{z,t}+|A-z|^{2}\right)}\right)\nonumber\\
    &\times f_{N,k}\int_{\mbf{M}_{k}}\exp\left\{-\frac{Nt}{2}\sum_{i=1}^{k}\left[h_{i}(|p_{i}|/\eta^{2}_{z,t})+h_{i}(|q_{i}|/\eta^{2}_{z,t})\right]\right\}dX,\label{eq:Fbound1}
\end{align}
where
\begin{align*}
    h_{i}(x)&=\frac{\eta^{4}_{z,t}}{t}\left\langle (H^{\left(k\right)}_{w_{i}})^{2}\right\rangle\frac{3x^{2}}{6+4x}-\frac{\eta^{2}_{z,t}}{t}{\Big|}\left\langle H^{\left(k\right)}_{w_{i}}\right\rangle-\left\langle H_{z}\right\rangle{\Big|}x,
\end{align*}
and $p_{i}=P_{ii}-\eta^{2}_{z,t}$, $q_{i}=Q_{ii}-\eta^{2}_{z,t}$.

For the product of ratios of determinants in front of the integral, we use Cramer's rule repeatedly to replace $A^{(i)}$ with $A^{(i-1)}$:
\begin{align}
    \frac{\det\left(\eta^{2}_{z,t}+|A^{(k)}-w_{i}|^{2}\right)}{\det\left(\eta^{2}_{z,t}+|A-z|^{2}\right)}&=\det\left(1_{2N}-\frac{1}{\sqrt{N\sigma_{z,t}}}G_{z}Z_{i}\right)\prod_{i=1}^{k}\left|\det V_{i}^{*}G^{(i-1)}_{w_{i}}V_{i}\right|,\label{eq:prefactor}
\end{align}
where we recall that $V_{i}=1_{2}\otimes\mbf{v}_{i}$. {We can truncate the Taylor series for the first factor after some large but finite $n$ terms using the fact that $\frac{1}{\sqrt{N\sigma_{z,t}}}\|G_{z}Z_{i}\|\leq CN^{-1/2}t^{-1}\leq N^{-\epsilon}$ (this also uses $\sigma_{z,t}>c>0$ which follows from \eqref{eq:A_alpha}):}
\begin{align*}
    \det\left(1-\frac{1}{\sqrt{N\sigma_{z,t}}}G_{z}Z_{i}\right)&=\exp\left\{-\sum_{m=1}^{n-1}\frac{1}{m(N\sigma_{z,t})^{m/2}}\tr\left((G_{z}Z_{i})^{m}\right)+O\left(\frac{1}{N^{n/2-1}t^{n}}\right)\right\}.
\end{align*}
Now we use \eqref{eq:trace2} to estimate each term in the sum:
\begin{align*}
    \det\left(1-\frac{1}{\sqrt{N\sigma_{z,t}}}G_{z}Z_{i}\right)&=\exp\left\{-\sqrt{\frac{N}{\sigma_{z,t}}}\left\langle G_{z}Z_{i}\right\rangle-\frac{1}{2\sigma_{z,t}}\left\langle(G_{z}Z_{i})^{2}\right\rangle+O\left(\frac{1}{\sqrt{Nt^{3}}}\right)\right\}\\
    &=O(\psi_{i}),
\end{align*}
where $\psi_{i}$ is defined in \eqref{eq:psi}. For the second equality we have used the fact that $|z_{i}|<C$ and $|\tau_{z,t}|<C$.

Each factor in the product on the right hand side of \eqref{eq:prefactor} can be estimated using \Cref{lem:detRatio}. {In the notation of \Cref{lem:detRatio}, with $V=V_{i},\,W^{-1}=-G^{(i-1)}_{w_{i}}$ and $Y^{-1}=-G^{(i-1)}_{z}$, we have
\begin{align*}
    \rho&=\|(\Im Y)^{-1}\|\cdot\|Y(Y^{-1}-W^{-1})Y\|\\
    &\leq\frac{1}{\eta_{z,t}}\|(G^{(i-1)}_{z})^{-1}(G^{(i-1)}_{z}-G^{(i-1)}_{w_{i}})(G^{(i-1)}_{z})^{-1}\|\\
    &\leq\frac{C}{\sqrt{Nt^{2}}},
\end{align*}
where we used the resolvent identity and the fact that $|z-w_{i}|=|z_{i}|/\sqrt{N\sigma_{z,t}}<CN^{-1/2}$. Thus by \eqref{eq:detRatio} we find
\begin{align}
    \frac{\det V_{i}^{*}G^{(i-1)}_{w_{i}}V_{i}}{\det V_{i}^{*}G^{(i-1)}_{z}V_{i}}&=1+O\left(\frac{1}{\sqrt{Nt^{2}}}\right).\label{eq:detRatio1}
\end{align}
}
We find the following bound for the product in \eqref{eq:Fbound1}:
\begin{align*}
    \prod_{i=1}^{k}\frac{\det\left(\eta^{2}_{z,t}+|A^{(k)}-w_{i}|^{2}\right)}{\det\left(\eta^{2}_{z,t}+|A-z|^{2}\right)}&\leq C\prod_{i=1}^{k}\left|\det V_{i}^{*}G^{(i-1)}_{z}V_{i}\right|^{k}\psi_{i}.
\end{align*}
If we move the factor $e^{Nk\phi_{z,t}}$ to the right hand side of \eqref{eq:Fbound1} we obtain
\begin{align*}
    |\wt{F}|&\leq C|\Delta(\mbf{z})|^{2}\left(\prod_{i=1}^{k}\left|\det V_{i}^{*}G^{(i-1)}_{z}V_{i}\right|^{k}\psi_{i}e^{-N\phi_{z,t}}\right)\\
    &\times f_{N,k}\int_{\mbf{M}_{k}}\exp\left\{-\frac{Nt}{2}\sum_{i=1}^{k}\left[h_{i}(|p_{i}|/\eta_{z,t}^{2})+h_{i}(|q_{i}|/\eta_{z,t}^{2})\right]\right\}dX.
\end{align*}
By interlacing of singular values, the resolvent identity and \eqref{eq:A2} we have
\begin{align*}
    \left\langle H^{\left(k\right)}_{w_{i}}\right\rangle-\left\langle H_{z}\right\rangle&=\frac{1}{2\eta_{z,t}}\Im\left\langle G_{w_{i}}-G_{z}\right\rangle+O\left(\frac{1}{Nt^{2}}\right)\\
    &=\frac{1}{2\eta_{z,t}}\Im\left(\frac{1}{\sqrt{N\sigma_{z,t}}}\left\langle G^{2}_{z}Z_{i}\right\rangle+\frac{1}{N\sigma_{z,t}}\left\langle(G_{z}Z_{i})^{2}G_{w_{i}}\right\rangle\right)+O\left(\frac{1}{Nt^{2}}\right)\\
    &=O\left(\frac{1}{Nt^{5/2}}\right).
\end{align*}
Likewise, by interlacing and \eqref{eq:A_gamma},
\begin{align*}
    \left\langle(H^{\left(k\right)}_{w_{i}})^{2}\right\rangle&=\left\langle H^{2}_{w_{i}}\right\rangle+O\left(\frac{1}{Nt^{4}}\right)\\
    &\geq\frac{c}{t^{3}}.
\end{align*}
Hence we obtain
\begin{align*}
    h_{i}\left(x\right)&\geq cx\left(\frac{x}{1+x}-\frac{1}{Nt^{3/2}}\right).
\end{align*}
In $\mbf{M}_{k}\setminus\Omega_{1}$, we have $x>\frac{\log N}{\sqrt{Nt}}$, which implies that $h_{i}\left(x\right)>\frac{\log^{2}N}{Nt}$ and the integrand is $O(e^{-c\log^{2}N})$. Since $f_{N,k}$ defined in \eqref{eq:fNk} is $O(N^{C})$ for some $C$, we conclude that
\begin{align}
    f_{N,k}|\Delta(\mbf{z})|^{2}\int_{\mbf{M}_{k}\setminus\Omega_{1}}e^{-\frac{N}{t}\tr|X|^{2}}\left|\det\left[iM^{(k)}(X)\right]\right|dX&\leq e^{-c\log^{2}N}.\label{eq:MminusOmega1}
\end{align}

Now we want to show that the integral over $\Omega_{1}\setminus\Omega_{2}$ is small and so we consider $X\in\Omega_{1}$. Returning to \eqref{eq:fischer1}, instead of bounding each determinant by the product of determinants of the diagonal blocks, we take a product of $k-2$ diagonal blocks and a $2\times 2$ block, i.e. we use the following special case of Fischer's inequality for a non-negative block matrix $B$:
\begin{align*}
    \det\begin{pmatrix}B_{11}&\cdots&B_{1k}\\\vdots&\ddots&\vdots\\B^{*}_{1k}&\cdots&B_{kk}\end{pmatrix}&\leq\det\begin{pmatrix}B_{ii}&B_{ij}\\B^{*}_{ij}&B_{jj}\end{pmatrix}\prod_{l\neq i,j}\det B_{ll}\\
    &=\det\left(1-B_{ii}^{-1}B_{ij}B_{jj}^{-1}B_{ij}^{*}\right)\prod_{i}\det B_{ii}.
\end{align*}
In our case we have {$B_{ii}=P_{ii}+|A^{(k)}-w_{i}|^{2}$ and $B_{ij}=P_{ij}1_{N-k}$}, leading to
\begin{align*}
    \frac{\det\left(XX^{*}\otimes1+|1\otimes A^{\left(k\right)}-\mbf{w}\otimes1|^{2}\right)}{\prod_{i=1}^{k}\det\left(P_{ii}+|A^{\left(k\right)}-w_{i}|^{2}\right)}&=\det\left(1_{N-k}-|P_{ij}|^{2}H^{\left(k\right)}_{w_{i}}\left(P_{ii}\right)H^{(k)}_{w_{j}}\left(P_{jj}\right)\right).
\end{align*}
In terms of $G_{w_{i}}$ and $G_{w_{j}}$ we have
\begin{align*}
    \langle H^{(k)}_{w_{i}}(P_{ii})H^{(k)}_{w_{j}}(P_{jj})\rangle&=-\frac{1}{P_{ii}P_{jj}}\langle G^{(k)}_{w_{i}}(P_{ii})E^{(k)}_{+}G^{(k)}_{w_{j}}(P_{jj})E^{(k)}_{+}\rangle,
\end{align*}
where $E_{+}$ was defined in \eqref{eq:rhoPlus}. {When $X\in\Omega_{1}$ we have} $P_{ii}>ct$ and $P_{jj}>ct$, so we can use \Cref{lem:minorresolvent} to replace $A^{(k)}$ with $A$:
\begin{align*}
    -\frac{1}{P_{ii}P_{jj}}\langle G^{(k)}_{w_{i}}(P_{ii})E^{(k)}_{+}G^{(k)}_{w_{j}}(P_{jj})E^{(k)}_{+}\rangle&=-\frac{1}{P_{ii}P_{jj}}\langle G_{w_{i}}(P_{ii})E_{+}G_{w_{j}}(P_{jj})E_{+}\rangle+O\left(\frac{1}{t^{4}}\right).
\end{align*}
By the resolvent identity,
\begin{align*}
    -\frac{1}{P_{ii}P_{jj}}\langle G_{w_{i}}(P_{ii})E_{+}G_{w_{j}}(P_{jj})E_{+}\rangle&=\langle H_{w_{i}}(P_{ii})H_{w_{i}}(P_{jj})\rangle\\&-\frac{1}{P_{ii}P_{jj}}\langle G_{w_{i}}(P_{ii})E_{+}G_{w_{i}}(P_{jj})W_{ij}G_{w_{j}}(P_{jj})E_{+}\rangle,
\end{align*}
where $W_{ij}=\begin{pmatrix}0&w_{j}-w_{i}\\\bar{w}_{j}-\bar{w}_{i}&0\end{pmatrix}$. The second term on the last line can be estimated by Cauchy-Schwarz, \eqref{eq:A_gamma}, \eqref{eq:A2} and the bound $\|W_{ij}\|\leq C/\sqrt{N}$:
\begin{align}
    \left|\langle G_{w_{i}}(P_{ii})E_{+}G_{w_{i}}(P_{jj})W_{ij}G_{w_{j}}(P_{jj})E_{+}\rangle\right|&\leq\langle G_{w_{i}}(P_{ii})E_{+}G^{*}_{w_{i}}(P_{ii})E_{+}\rangle^{1/2}\nonumber\\&\times\left\langle\frac{\Im G_{w_{i}}(P_{jj})}{P_{jj}}W_{ij}\frac{\Im G_{w_{j}}(P_{jj})}{P_{jj}}W^{*}_{ij}\right\rangle^{1/2}\nonumber\\
    &\leq\frac{C}{N^{1/2}t^{3/2}}.\label{eq:error}
\end{align}
Since $|P_{ll}-\eta^{2}_{z,t}|<\sqrt{\frac{t^{3}}{N}}\log N$ for each $l=1,...,k$, we have
\begin{align*}
    \langle H_{w_{i}}(P_{ii})H_{w_{i}}(P_{jj})\rangle&\geq\frac{1}{1+|P_{jj}-P_{ii}|/\eta_{z,t}^{2}}\langle H^{2}_{w_{i}}(P_{ii})\rangle\\
    &\geq\frac{c}{t^{3}},
\end{align*}
where the last line follows by interlacing and \eqref{eq:A_gamma}. Combining this lower bound with the bound in \eqref{eq:error} we obtain
\begin{align*}
    \langle H^{(k)}_{w_{i}}(P_{ii})H^{(k)}_{w_{j}}(P_{jj})\rangle&\geq\frac{c}{t^{3}}.
\end{align*}
From this and the inequality $\det\left(1+B\right)\leq e^{\tr B}$ we observe that if $|P_{ij}|\geq\sqrt{\frac{t^{3}}{N}}\log N$ then
\begin{align*}
    \det\left(1_{N-k}-|P_{ij}|^{2}H^{(k)}_{w_{i}}(P_{ii})H^{(k)}_{w_{j}}(P_{jj})\right)&\leq e^{-c\log^{2}N}.
\end{align*}
We can repeat this for all pairs of indices $\left(i,j\right)$ to conclude that
\begin{align}
    f_{N,k}|\Delta(\mbf{z})|^{2}\int_{\Omega_{1}\setminus\Omega_{2}}e^{-\frac{N}{t}\tr|X|^{2}}\left|\det\left[iM^{(k)}(X)\right]\right|dX\leq e^{-c\log^{2}N}.\label{eq:Omega1minusOmega2}
\end{align}
Combining \eqref{eq:Omega1minusOmega2} and \eqref{eq:MminusOmega1}, we obtain \eqref{eq:MminusOmega'} and the lemma is proved.
\end{proof}

{Having restricted the integral to $\Omega$, we are now able to approximate the integrand by Taylor expansion.} {For $l=0,...,k$} define the $k\times k$ block matrices $\mbf{Z}^{(l)}$ with blocks
\begin{align*}
    Z^{(l)}_{ij}&=\begin{pmatrix}0&(U_{1}^{*}\mbf{z}U_{1})_{ij}1_{N-l}\\(U_{2}^{*}\bar{\mbf{z}}U_{2})_{ij}1_{N-l}&0\end{pmatrix},
\end{align*}
and the block resolvents $\mbf{F}^{\left(l\right)}_{z}\left(\mbf{s}\right)=\left(1_{k}\otimes H^{\left(l\right)}_{z}-i\mbf{s}\otimes1_{2(N-l)}-\frac{1}{\sqrt{N\sigma_{z,t}}}\mbf{Z}^{\left(l\right)}\right)^{-1}$ and $\mbf{G}\left(\mbf{s}\right)=(1_{k}\otimes H_{z}-i\mbf{s}\otimes1_{2N})^{-1}$. When $l=0$ we drop the superscript, i.e. $\mbf{Z}:=\mbf{Z}^{(0)}$ and so on. We also continue the convention of dropping the argument when $\mbf{s}=\eta_{z,t}1_{k}$, i.e. $\mbf{G}^{(j)}:=\mbf{G}^{(j)}(\eta_{z,t}1_{k})$ and $\mbf{F}^{(j)}:=\mbf{F}^{(j)}(\eta_{z,t}1_{k}))$. The following lemma gives the approximation of the integrand.
{\begin{lemma}
Let $X\in\Omega$ have singular value decomposition $U_{1}\mbf{s}U_{2}^{*}$; then we have
\begin{align}
    e^{-\frac{N}{t}\tr|X|^{2}}\det\left[iM^{(k)}(X)\right]&=\left[1+O\left(\frac{\log N}{\sqrt{Nt^{3}}}\right)\right]\left(\prod_{i=1}^{k}|\det V_{i}^{*}G^{(i-1)}_{z}V_{i}|^{k}\psi_{i}e^{-N\phi_{z,t}}\right)\nonumber\\&\times\exp\left\{\sum_{i=1}^{k}\left(\Re\left(\frac{\bar{\beta}^{2}_{z,t}z^{2}_{i}}{\gamma_{z,t}\sigma_{z,t}}\right)-|z_{i}|^{2}\right)+\frac{\alpha_{z,t}}{\sigma_{z,t}}\tr (U_{1}^{*}\mbf{z}U_{1}U_{2}^{*}\bar{\mbf{z}}U_{2})\right.\nonumber\\&\left.-2N\gamma_{z,t}\sum_{j=1}^{k}(s_{j}-\eta_{z,t})^{2}-i\sqrt{\frac{N}{\sigma_{z,t}}}\sum_{j=1}^{k}\Tr{G^{2}_{z}Z_{jj}}(s_{j}-\eta_{z,t})\right\},
\end{align}
uniformly in $\mbf{v}_{1},...,\mbf{v}_{k}$. If $\eqref{eq:A3}$ holds then the error is $O\left(\frac{\log N}{\sqrt{Nt^{2}}}\right)$.
\end{lemma}}
\begin{proof}
{In the following the constants implied by $O(\cdot)$ are uniformly bounded in $\mbf{v}_{1},...,\mbf{v}_{k}$.} By successive application of Cramer's rule to replace $A^{\left(i\right)}$ with $A^{\left(i-1\right)}$ we obtain
\begin{align}
    \det M^{\left(k\right)}(X)&=\det\left(1_{2kN}-\frac{1}{\sqrt{N\sigma_{z,t}}}\mbf{G}\left(\mbf{s}\right)\mbf{Z}\right)\prod_{j=1}^{k}\det\left(\mbf{V}_{j}^{*}\mbf{F}^{\left(j-1\right)}_{z}\left(\mbf{s}\right)\mbf{V}_{j}\right)\det\left(\mc{H}_{z}-is_{j}\right),\label{eq:detM}
\end{align}
where $\mbf{V}_{i}=1_{k}\otimes V_{i}$. We start by proving the claim
\begin{align}
    e^{-\frac{N}{t}s_{j}^{2}}\det(\mc{H}_{z}-is_{j})&=\left[1+O\left(\frac{\log N}{\sqrt{Nt}}\right)\right]e^{-\frac{N}{t}\eta_{z,t}^{2}}\det(\mc{H}_{z}-i\eta_{z,t})\cdot e^{-2N\gamma_{z,t}(s-\eta_{z,t})^{2}}.\label{eq:claim1}
\end{align}

Since $|s_{j}-\eta_{z,t}|<\sqrt{\frac{t}{N}}\log N$ in $\Omega$ and $\|G_{z}\|\leq\eta^{-1}_{z,t}$, we can use the integral representation of the logarithm to obtain
\begin{align*}
    \frac{\det\left(\mathcal{H}_{z}-is\right)}{\det\left(\mathcal{H}_{z}-i\eta_{z,t}\right)}&=\det\left[1_{2N}-i\left(s-\eta_{z,t}\right) G_{z}\left(\eta_{z,t}\right)\right]\\
    &=\exp\left\{\left(s-\eta_{z,t}\right)\Im\left\langle G_{z}\right\rangle+\frac{1}{2}\left(s-\eta_{z,t}\right)^{2}\Re\tr G_{z}^{2}\right.\\&\left.-\left(s-\eta_{z,t}\right)^{3}\int_{0}^{1}u^{2}\Im\tr G^{3}_{z}(1-iu\left(s-\eta_{z,t}\right) G_{z})^{-1}du\right\}.
\end{align*}
We can bound the last term as follows:
\begin{align*}
    |s-\eta_{z,t}|^{3}\left|\int_{0}^{1}u^{2}\Im\tr(1_{2N}-iu\left(s-\eta_{z,t}\right)G_{z})^{-1}G^{3}_{z}du\right|&\leq C\sqrt{\frac{t}{N}}\left\langle|G_{z}|^{2}\right\rangle\log^{3}N\\
    &\leq\frac{c\log^{3}N}{\sqrt{Nt}}.
\end{align*}
The first two terms, when combined with the factor $e^{-\frac{N}{t}\left(s^{2}-\eta^{2}_{z,t}\right)}$, give the exponent
\begin{align*}
    -\frac{N}{t}\left[s^{2}-\eta_{z,t}^{2}-2\eta_{z,t}\left(s-\eta_{z,t}\right)-\frac{t}{2}\left(s-\eta_{z,t}\right)^{2}\Re\left\langle G^{2}_{z}\right\rangle\right]&=-2N\gamma_{z,t}\left(s-\eta_{z,t}\right)^{2},
\end{align*}
which proves the claim \eqref{eq:claim1}.

{Now we come to the first factor in \eqref{eq:detM}.} Since $\|\mbf{G}\left(\mbf{s}\right)\|<C/t$ and $Nt^{2}=N^{-2\epsilon}$, we can find a finite $n$ such that $(Nt^{2})^{n/2}\gg N$. Then we can truncate the Taylor series expansion of $\det\left(1-N^{-1/2}\mbf{G}\left(\mbf{s}\right)\mbf{Z}\right)$ at finite order:
\begin{align*}
    \det\left(1_{2kN}-\frac{1}{\sqrt{N\sigma_{z,t}}}\mbf{G}\left(\mbf{s}\right)\mbf{Z}\right)&=\exp\left\{-\sum_{m=1}^{n-1}\frac{1}{m(N\sigma_{z,t})^{m/2}}\tr\left((\mbf{G}\left(\mbf{s}\right)\mbf{Z})^{m}\right)+O\left(\frac{1}{N^{n/2-1}t^{n}}\right)\right\}.
\end{align*}
By \eqref{eq:trace2}, we have
\begin{align*}
    \left|\frac{1}{(N\sigma_{z,t})^{m/2}}\tr\left((\mbf{G}\left(\mbf{s}\right)\mbf{Z})^{m}\right)\right|&\leq\frac{N}{(N\sigma_{z,t})^{m/2}}\sum\left|\left\langle G(s_{i_{1}})Z_{i_{1}i_{2}}\cdots G\left(s_{i_{m}}\right)Z_{i_{m}i_{1}}\right\rangle\right|\\
    &\leq\begin{cases}C/\sqrt{Nt^{3}}&\quad m=3\\
    C/\left(Nt^{2}\right)^{m/2-1}&\quad m\geq4
    \end{cases}\\
    &\leq\frac{C}{\sqrt{Nt^{3}}}.
\end{align*}
Note that if $A$ satisfies \eqref{eq:A3} then the $m=3$ term is $O\left(1/\sqrt{Nt^{2}}\right)$ so we obtain a better error bound. Using the resolvent identity $G_{z}\left(s_{j}\right)=(1+i\left(s_{j}-\eta_{z,t}\right)G_{z}\left(s_{j}\right))G_{z}$ and the local law we have
\begin{align*}
    \left\langle G_{z}\left(s_{j}\right)Z_{jj}\right\rangle&=\left\langle G_{z}Z_{jj}\right\rangle+i(s_{j}-\eta_{z,t})\left\langle G^{2}_{z}Z_{jj}\right\rangle-\left(s_{j}-\eta_{z,t}\right)^{2}\left\langle G_{z}\left(s_{j}\right)G^{2}_{z}Z_{jj}\right\rangle\\
    &=\left\langle G_{z}Z_{jj}\right\rangle+i(s_{j}-\eta_{z,t})\left\langle G^{2}_{z}Z_{jj}\right\rangle+O\left(\frac{\log^{2}N}{N\sqrt{t}}\right),
\end{align*}
and similarly
\begin{align*}
    \left\langle G_{z}\left(s_{i}\right)Z_{ij}G_{z}\left(s_{j}\right)Z_{ji}\right\rangle&=\left\langle G_{z}Z_{ij}G_{z}Z_{ji}\right\rangle+O\left(\frac{\log N}{\sqrt{Nt^{2}}}\right).
\end{align*}
Thus we find
\begin{align}
    \det\left(1_{2kN}-\frac{1}{\sqrt{N\sigma_{z,t}}}\mbf{G}\left(\mbf{s}\right)\mbf{Z}\right)&=\exp\left\{-\sqrt{\frac{N}{\sigma_{z,t}}}\sum_{i=1}^{k}\Tr{G_{z}Z_{ii}}-\frac{1}{2\sigma_{z,t}}\sum_{i,j=1}^{k}\Tr{G_{z}Z_{ij}G_{z}Z_{ji}}\right.\nonumber\\
    &\left.-{i}\sqrt{\frac{N}{\sigma_{z,t}}}\sum_{j=1}^{k}\Tr{G^{2}_{z}Z_{jj}}\left(s_{j}-\eta_{z,t}\right)+O\left(\frac{1}{\sqrt{Nt^{3}}}\right)\right\}.\label{eq:claim2}
\end{align}
From the definition of $Z_{ij}$ we observe that
\begin{align*}
    \sum_{i=1}^{k}\Tr{G_{z}Z_{ii}}&=\Tr{G_{z}\begin{pmatrix}0&\tr(\mbf{z})\\\tr(\bar{\mbf{z}})&0\end{pmatrix}}=\sum_{i=1}^{k}\Tr{G_{z}Z_{i}},
\end{align*}
and
\begin{align*} 
    -\frac{1}{2\sigma_{z,t}}\sum_{i,j=1}^{k}\left\langle G_{z}Z_{ij}G_{z}Z_{ji}\right\rangle&=\frac{\alpha_{z,t}}{\sigma_{z,t}}\tr (U_{1}^{*}\mbf{z}U_{1}U_{2}^{*}\bar{\mbf{z}}U_{2})-\sum_{i=1}^{k}\Re\left(\frac{\bar{\delta}_{z,t}z_{i}^{2}}{\sigma_{z,t}}\right).
\end{align*}

{For the product over $j=1,...,k$ in \eqref{eq:detM}, we use \Cref{lem:detRatio} with $V=\mbf{V}_{j}$, $Y^{-1}=-\mbf{F}^{(j-1)}$ and $W^{-1}=-\mbf{F}^{(j-1)}(\mbf{s})$. Note that $\Im Y=\eta_{z,t}+\frac{1}{\sqrt{N\sigma_{z,t}}}\Im\mbf{Z}^{(j-1)}\geq Ct$. Since in $\Omega$ we have $\|\mbf{s}-\eta_{z,t}\|\leq\sqrt{\frac{t}{N}}\log N$, by the resolvent identity we obtain
\begin{align*}
    \rho&=\|(\Im Y)^{-1}\|\cdot\|Y(Y^{-1}-W^{-1})Y\|\\
    &\leq\frac{C\log N}{\sqrt{Nt^{2}}}.
\end{align*}
By \eqref{eq:detRatio} we now have
\begin{align}
    \frac{\det\mbf{V}_{j}^{*}\mbf{F}^{(j-1)}(\mbf{s})\mbf{V}_{j}}{\det\mbf{V}_{j}^{*}\mbf{F}^{(j-1)}\mbf{V}_{j}}&=1+O\left(\frac{\log N}{\sqrt{Nt^{2}}}\right),\label{eq:FstoF}
\end{align}
and
\begin{align*}
    \frac{\det\mbf{V}_{i}^{*}\mbf{F}^{\left(i-1\right)}\mbf{V}_{i}}{\det\mbf{V}_{i}^{*}\mbf{G}^{\left(i-1\right)}\mbf{V}_{i}}&=1+O\left(\frac{\log N}{\sqrt{Nt^{2}}}\right),
\end{align*}
and thus
\begin{align}
    \prod_{j=1}^{k}\det\mbf{V}_{j}^{*}\mbf{F}^{(j-1)}(\mbf{s})\mbf{V}_{j}&=\left[1+O\left(\frac{\log N}{\sqrt{Nt^{2}}}\right)\right]\prod_{j=1}^{k}\det^{k}\left(V_{j}^{*}G^{(j-1)}_{z}V_{j}\right).\label{eq:claim3}
\end{align}
Combining \eqref{eq:claim1}, \eqref{eq:claim2} and \eqref{eq:claim3} and recalling the definition of $\psi_{i}$ from \eqref{eq:psi}, we conclude the proof.
}
\end{proof}
{We take the error outside the integral; in principle one needs to show that the integral of the absolute value of the leading order term is bounded above by a constant multiple of the absolute value of the integral, but this will be clear from the ensuing calculations. Thus, we have
\begin{align*}
    \wt{F}(\mbf{z};A^{(k)})&=\left[1+O\left(\frac{\log N}{\sqrt{Nt^{3}}}\right)\right]\wt{F}_{0}(\mbf{z};A^{(k)}),
\end{align*}
where $\wt{F}_{0}$ is given by
\begin{align*}
    \wt{F}_{0}(\mbf{z};A^{(k)})&=\frac{f_{N,k}v_{k}|\Delta\left(\mbf{z}\right)|^{2}\eta_{z,t}^{k^{2}}}{2^{k}\gamma_{z,t}^{k^{2}/2}}\left(\prod_{i=1}^{k}|\det V_{i}^{*}G^{\left(i-1\right)}_{z}V_{i}|^{k}\psi_{i}e^{-N\phi_{z,t}}\right)\\&\times e^{-\sum_{i=1}^{k}|z_{i}|^{2}}\int_{G}e^{\frac{\alpha_{z,t}}{\sigma_{z,t}}\tr(U_{1}^{*}\mbf{z}U_{1}U_{2}^{*}\bar{\mbf{z}}U_{2})}f(U_{1},U_{2})d\mu\left(U_{1},U_{2}\right).
\end{align*}}
Here {$G=U(k)\times U(k)/U(1)^{k}$, $d\mu(U_{1},U_{2})$ is the normalised Haar measure on $G$, $v_{k}=\text{Vol}(U(k))$ and}
\begin{align*}
    f\left(U_{1},U_{2}\right)&=\frac{v_{k}}{k!\left(2\pi\right)^{k}}e^{\sum_{i=1}^{k}\Re\left(\frac{\bar{\beta}^{2}_{z,t}z_{i}^{2}}{\gamma_{z,t}\sigma_{z,t}}\right)}\int_{\mbb{R}^{k}}e^{-\frac{1}{2}\sum_{i=1}^{k}s_{i}^{2}-{i}\sqrt{\frac{t}{\sigma_{z,t}}}\sum_{j=1}^{k}\left\langle G^{2}_{z}Z_{jj}\right\rangle s_{j}}\Delta^{2}\left(\mbf{s}\right)d\mbf{s}.
\end{align*}
To obtain this form we have changed variables $X\mapsto U_{1}\left(\eta_{z,t}+\mbf{s}/\sqrt{4N\gamma_{z,t}}\right)U_{2}^{*}$ {with Jacobian
\begin{align*}
    dX&=\frac{v_{k}^{2}\Delta^{2}(\mbf{s})}{k!(2\pi)^{k}(4N\gamma_{z,t})^{k(k-1)/2}}\prod_{i<j}^{k}\left(2\eta_{z,t}+\frac{s_{i}+s_{j}}{\sqrt{4N\gamma_{z,t}}}\right)\prod_{i=1}^{k}\left(\eta_{z,t}+\frac{s_{i}}{\sqrt{4N\gamma_{z,t}}}\right)ds_{i}d\mu(U_{1},U_{2}),
\end{align*}}
and extended the integrals over $s_{i}$ to the whole real line (at the cost of an $O\left(e^{-{c\log^{2}N}}\right)$ error). 

To compute $f(U_{1},U_{2})$ we first note that
\begin{align*}
    \frac{\alpha_{z,t}}{\sigma_{z,t}}\tr(U_{1}^{*}\mbf{z}U_{1}U_{2}^{*}\bar{\mbf{z}}U_{2})
\end{align*}
depends only on $U_{1}$ and $U_{2}$ in the combination $U_{1}U_{2}^{*}$, which is invariant under the transformation $\left(U_{1},U_{2}\right)\mapsto\left(U_{1}V,U_{2}V\right)$. The measure $d\mu$ is also invariant under this transformation. Therefore we can replace $f\left(U_{1},U_{2}\right)$ with its average
\begin{align*}
    \wt{f}(U_{1},U_{2})&=\int_{U\left(k\right)}f(U_{1}V,U_{2}V)d\mu\left(V\right).
\end{align*}
The term depending on $U_{1}$ and $U_{2}$ in the integrand of $f$ is
\begin{align*}
    \sum_{j=1}^{k}\left\langle G_{z}^{2}Z_{jj}\right\rangle s_{j}&=2i\sum_{j=1}^{k}\left(\beta_{z,t}\left(U^{{*}}_{2}\bar{\mbf{z}}U_{2}\right)_{jj}-\bar{\beta}_{z,t}\left(U^{{*}}_{1}\mbf{z}U_{1}\right)_{jj}\right)s_{j}.
\end{align*}
Now we note that $\frac{v_{k}}{k!\left(2\pi\right)^{k}}\Delta^{2}\left(\mbf{s}\right)d\mbf{s}d\mu\left(V\right)$ is the Jacobian of the map $Q=V\mbf{s}V^{*}$, so that we can write
\begin{align*}
    \wt{f}(U_{1},U_{2})&=e^{\Re\sum_{i=1}^{k}\frac{\bar{\beta}^{2}_{z,t}z_{i}^{2}}{\gamma_{z,t}\sigma_{z,t}}}\int_{\mbf{M}^{sa}_{k}}e^{-\frac{1}{2}\tr Q^{2}+\frac{{1}}{\sqrt{\gamma_{z,t}\sigma_{z,t}}}\tr\left[\bar{\beta}_{z,t}\left(U_{1}^{*}\mbf{z}U_{1}\right)-\beta_{z,t}\left(U_{2}^{*}\bar{\mbf{z}}U_{2}\right)\right]Q}dQ\\
    &=2^{k/2}\pi^{k^{2}/2}e^{\frac{|\beta_{z,t}|^{2}}{\gamma_{z,t}\sigma_{z}}{\tr}(U_{1}^{*}\mbf{z}U_{1}U_{2}^{*}\bar{\mbf{z}}U_{2})}.
\end{align*}
Inserting this into the expression for $\wt{F}_{0}$ {and recalling the definition $\sigma_{z,t}=\alpha_{z,t}+\frac{|\beta_{z,t}|^{2}}{\gamma_{z,t}}$}, we find that we are left with evaluating the following integral over $G$:
\begin{align*}
    I&=\int_{G}e^{\tr (U^{{*}}_{1}ZU_{1}U_{2}^{*}Z^{*}U_{2})}d\mu\left(U_{1},U_{2}\right)\\
    &=\int_{U\left(k\right)}e^{\tr (UZU^{*}Z^{*})}d\mu\left(U\right)\\
    &=\left(\prod_{j=1}^{k-1}j!\right)\frac{\det\left[e^{z_{i}\bar{z}_{j}}\right]}{|\Delta\left(\mbf{z}\right)|^{2}}.
\end{align*}
The final equality is the Harish-Chandra--Itzykson--Zuber formula. We find that the constant prefactors combine to result in
\begin{align*}
    \frac{f_{N,k}\pi^{k^{2}/2}v_{k}\left(\prod_{j=1}^{k-1}j!\right)\eta^{k^{2}}_{z,t}}{2^{k/2}\gamma_{z,t}^{k^{2}/2}}&=\frac{1}{\pi^{k}}\left(\frac{t^{3}\gamma_{z,t}}{\eta_{z,t}^{2}}\right)^{k}\left(\frac{\eta^{2}_{z,t}}{t^{2}\gamma_{z,t}\sigma_{z,t}}\right)^{k(k+1)/2}.
\end{align*}
Putting everything together we obtain \eqref{eq:Fasymptotic}.

\section{Integral over $S^{N-i}$}\label{sec:K}
The analysis of the integrals over $\mbf{v}_{i}$ follows the same pattern as the previous section. First we derive an integral formula in which $N$ appears as a parameter, then we use Laplace's method to obtain an asymptotic approximation.
\begin{lemma}
Let $\wt{K}_{j}$ be as defined in \eqref{eq:Ktilde} and $\eta>0$. Then
\begin{align}
    \wt{K}_{j}\left(z_{j};A^{\left(j-1\right)}\right)&=\frac{1}{t}\sqrt{\frac{N}{2\pi t}}\frac{e^{\frac{N}{t}\eta^{2}}}{\det\left(\eta^{2}+|A^{\left(j-1\right)}-w_{j}|^{2}\right)}\int_{-\infty}^{\infty}e^{\frac{iNp}{t}}\det^{-1}(1+ipH^{\left(j-1\right)}_{w_{j}}\left(\eta\right))dp,\label{eq:Kformula}
\end{align}
{where we recall that $H^{(j-1)}_{w_{j}}(\eta)=(\eta^{2}+|A^{(j-1)}-w_{j}|^{2})^{-1}$.}
\end{lemma}
\begin{proof}
We begin by multiplying by $e^{\frac{N}{t}\eta^{2}}$ and writing $\wt{K}_{j}$ in such a way that we can apply \Cref{lem:sphericalint}:
\begin{align*}
    \wt{K}_{j}\left(z_{j};A^{\left(j-1\right)}\right)&=e^{\frac{N}{t}\eta^{2}}\int_{S^{N-j}}f\left(\mbf{v}_{j}\right)dS_{N-j}\left(\mbf{v}_{j}\right),
\end{align*}
where
\begin{align*}
    f\left(\mbf{v}_{j}\right)&=\left(\frac{N}{\pi t}\right)^{N-j}e^{-\frac{N}{t}\mbf{v}_{j}^{*}\left(\eta^{2}+|A^{\left(j-1\right)}-w_{j}|^{2}\right)\mbf{v}_{j}}.
\end{align*}
The transform $\hat{f}$ is a Gaussian integral:
\begin{align*}
    \frac{N}{t}\hat{f}\left(\frac{Np}{t}\right)&=\left(\frac{N}{\pi t}\right)^{N-j+1}\int_{\mbb{C}^{N-j+1}}e^{-\frac{iNp}{t}\left\|\mbf{x}\right\|^{2}}f\left(\mbf{x}\right)d\mbf{x}\\
    &=\det^{-1}\left(\eta^{2}+ip+|A^{\left(j-1\right)}-w_{j}|^{2}\right).
\end{align*}
Since $\hat{f}\in L^{1}$, we can apply \Cref{lem:sphericalint}, after which we are left with an integral over a single real variable:
\begin{align*}
    \wt{K}_{j}(z_{j};A^{\left(j-1\right)})&=\frac{1}{t}\sqrt{\frac{N}{2\pi t}}e^{\frac{N}{t}\eta^{2}}\int_{-\infty}^{\infty}\frac{N}{t}\hat{f}\left(\frac{Np}{t}\right)dp\\
    &=\frac{1}{t}\sqrt{\frac{N}{2\pi t}}\frac{e^{\frac{N}{t}\eta^{2}}}{\det\left(\eta^{2}+|A^{\left(j-1\right)}-w_{j}|^{2}\right)}\int_{-\infty}^{\infty}e^{\frac{iNp}{t}}\det^{-1}(1+ipH^{\left(j-1\right)}_{w_{j}}\left(\eta\right))dp.
\end{align*}
\end{proof}

In the above lemma we are free to choose $\eta$ and so we set $\eta=\eta_{z,t}$. {We first claim that we can neglect the integral over $|p|>\sqrt{\frac{t^{3}}{N}}\log N$, i.e.
\begin{align}
    \left|\int_{|p|>\sqrt{\frac{t^{3}}{N}}\log N}e^{\frac{iNp}{t}}\det^{-1}(1+ipH^{(j-1)}_{w_{j}})dp\right|&\leq e^{-c\log^{2}N}.\label{eq:large_p}
\end{align}
To prove this, observe that by \eqref{eq:A_gamma} {and interlacing} we have $\Tr{(H^{(j-1)}_{w_{j}})^{2}}\geq cN/t^{3}$. Using the bound
\begin{align*}
    |\det^{-1}(1+ipH^{(j-1)}_{w_{j}})|&\leq\exp\left\{-\frac{Np^{2}\Tr{(H^{(j-1)}_{w_{j}})^{2}}}{2(1+p^{2}/\eta_{z,t}^{4})}\right\}
\end{align*}
when $|p|<\eta_{z,t}^{2}+\|A^{(j-1)}_{w_{j}}\|^{2}$ (this follows from $\log(1+x)\geq\frac{x}{1+x}$) and the bound
\begin{align*}
    |\det^{-1}(1+ipH^{(j-1)}_{w_{j}})|&\leq\left(1+\frac{p^{2}}{(\eta_{z,t}^{2}+\|A^{(j-1)}_{w_{j}}\|^{2})^{2}}\right)^{-N/2}
\end{align*}
otherwise, we deduce that the integrand is $O(e^{-c\log^{2}N})$ in the region $\sqrt{\frac{t^{3}}{N}}\log N<|p|<\eta^{2}_{z,t}+\|A^{(j-1)}_{w_{j}}\|^{2}$ and decays as $|p|^{-N}$ for large $|p|$, from which \eqref{eq:large_p} follows.}

In the region $|p|<\sqrt{\frac{t^{3}}{N}}\log N$ we Taylor expand the determinant in the integrand:
\begin{align*}
    \det^{-1}(1+ipH^{\left(j-1\right)}_{w_{j}})&=\exp\left\{-ip\tr H^{\left(j-1\right)}_{w_{j}}-\frac{1}{2}p^{2}\tr(H^{\left(j-1\right)}_{w_{j}})^{2}+O\left(\frac{\log^{3}N}{\sqrt{Nt}}\right)\right\}.
\end{align*}
As in the previous section, we use interlacing and the resolvent identity to replace $G^{\left(j-1\right)}_{w_{j}}$ with $G_{z}$. For the first term we find
\begin{align*}
    \left\langle H^{\left(j-1\right)}_{w_{j}}\right\rangle&=\left\langle H_{w_{j}}\right\rangle+O\left(\frac{1}{Nt^{2}}\right)\\
    &=\frac{1}{2\eta_{z,t}}\Im\left\langle G_{z}\right\rangle+\frac{1}{2\sqrt{N\sigma_{z,t}}\eta_{z,t}}\Im\left\langle G^{2}_{z}Z_{j}\right\rangle+\frac{1}{N\sigma_{z,t}\eta_{z,t}}\Im\left\langle\left(G_{z}Z_{j}\right)^{2}G_{w_{j}}\right\rangle+O\left(\frac{1}{Nt^{2}}\right)\\
    &=\left\langle H_{z}\right\rangle+\frac{1}{2\sqrt{N\sigma_{z,t}}\eta_{z,t}}\left\langle G^{2}_{z}Z_{j}\right\rangle+O\left(\frac{1}{Nt^{5/2}}\right).
\end{align*}
For the second term we find
\begin{align}
    \left\langle(H^{\left(j-1\right)}_{w_{j}})^{2}\right\rangle&=\left\langle H^{2}_{w_{j}}\right\rangle+O\left(\frac{1}{Nt^{4}}\right)\nonumber\\
    &=\frac{1}{4\eta_{z,t}^{3}}\left\langle\Im G_{w_{j}}-\eta_{z,t} G^{2}_{w_{j}}\right\rangle+O\left(\frac{1}{Nt^{4}}\right)\nonumber\\
    &=\left\langle H^{2}_{z}\right\rangle+O\left(\frac{1}{\sqrt{N}t^{4}}\right).\label{eq:gammaPert}
\end{align}
Thus we obtain
\begin{align*}
    \det^{-1}(1+ipH^{\left(j-1\right)}_{w_{j}})&=\exp\left\{-\frac{iNp}{t}-\frac{ip}{2\sqrt{N}\eta_{z,t}}\tr G^{2}_{z}Z_{j}-\frac{1}{2}p^{2}\tr H^{2}_{z}+O\left(\frac{\log N}{\sqrt{Nt^{2}}}\right)\right\}.
\end{align*}
{After replacing the integrand by the above approximation and taking the error outside the integral,} the integral over $p$ can be extended to the whole real line. {This leaves us with a Gaussian integral that is easily computed:}
\begin{align}
    \int_{-\infty}^{\infty}e^{\frac{iNp}{t}}\det^{-1}\left(1+ipH^{\left(j-1\right)}_{w_{j}}\right)dp&=\left[1+O\left(\frac{\log N}{\sqrt{Nt^{2}}}\right)\right]\sqrt{\frac{2\pi\eta_{z,t}^{2}}{N\gamma_{z,t}}}\exp\left\{-\frac{\left\langle G^{2}_{z}Z_{j}\right\rangle^{2}}{8\gamma_{z,t}\sigma_{z,t}}\right\}.\label{eq:pIntegral}
\end{align}

{Now consider the determinant in front of the integral in \eqref{eq:Kformula}.} From the discussion following \eqref{eq:prefactor}, we have
\begin{align*}
    \frac{\det\left(\eta_{z,t}^{2}+|A-z|^{2}\right)}{\det\left(\eta_{z,t}^{2}+|A^{(j-1)}-w_{j}|^{2}\right)}&=\left[1+O\left(\frac{1}{\sqrt{Nt^{2}}}\right)\right]\left(\prod_{l=1}^{j-1}\left|\det^{-1}(V_{l}^{*}G^{(l-1)}_{z}V_{l})\right|\right)\\&\times\exp\left\{\sqrt{\frac{N}{\sigma_{z,t}}}\left\langle G_{z}Z_{j}\right\rangle+\frac{1}{2\sigma_{z,t}}\left\langle(G_{z}Z_{j})^{2}\right\rangle\right\}.
\end{align*}
Note that
\begin{align*}
    -\frac{1}{8}\left\langle G^{2}_{z}Z_{j}\right\rangle^{2}&=\Re\left(\bar{\beta}_{z,t}^{2}z_{j}^{2}\right)-\left|\beta_{z,t}z_{j}\right|^{2},
\end{align*}
and
\begin{align*}
    \frac{1}{2}\left\langle(G_{z}Z_{j})^{2}\right\rangle&=\Re(\bar{\delta}_{z,t}z_{j}^{2})-\alpha_{z,t}|z_{j}|^{2},
\end{align*}
so that
\begin{align}
    &\frac{\det(\eta^{2}_{z,t}+|A-z|^{2})}{\det(\eta^{2}_{z,t}+|A^{(j-1)}-w_{j}|^{2})}\nonumber\\&=\left[1+O\left(\frac{1}{\sqrt{Nt^{2}}}\right)\right]\left(\prod_{l=1}^{j-1}|\det^{-1}(V_{l}^{*}G^{(l-1)}_{z}V_{l})|\right)e^{\Re(\bar{\tau}_{z,t}z_{j}^{2})-|z_{j}|^{2}}.\label{eq:prefactor2}
\end{align}
Combining \eqref{eq:pIntegral} and \eqref{eq:prefactor2} we obtain the asymptotic in \eqref{eq:Kasymptotic}.

{To prove \eqref{eq:qformexpectation}, we first note that by \eqref{eq:detRatio1} we have
\begin{align*}
    \frac{\det V_{j}^{*}G^{(j-1)}_{z}V_{j}}{\det V_{j}^{*}G^{(j-1)}_{w_{j}}V_{j}}&=1+O\left(\frac{1}{\sqrt{Nt^{2}}}\right),
\end{align*}}
and can therefore replace $V_{j}^{*}G^{\left(j-1\right)}_{z}V_{j}$ with $V_{j}^{*}G^{\left(j-1\right)}_{w_{j}}V_{j}$.

Secondly, we note that
\begin{align*}
    \left|\det V_{j}^{*}G^{\left(j-1\right)}_{w_{j}}V_{j}\right|&=\eta^{2}_{z,t}\left(\mbf{v}_{j}^{*}H^{(j-1)}_{w_{j}}\mbf{v}_{j}\right)\left(\mbf{v}_{j}^{*}\wt{H}^{(j-1)}_{w_{j}}\mbf{v}\right)+\left|\mbf{v}_{j}^{*}H^{(j-1)}_{w_{j}}(A^{(j-1)}-w_{j})\mbf{v}_{j}\right|^{2}.
\end{align*}
We will show that each of these quadratic forms concentrates around a deterministic quantity with high probability, and hence the integral on the left-hand side of \eqref{eq:qformexpectation} can be evaluated by simply replacing the integrand by a constant. The relevant quantities are
\begin{align}
    \wt{\alpha}^{(j-1)}&=\eta_{z,t}^{2}\left\langle\wt{H}^{(j-1)}_{w_{j}}H^{(j-1)}_{w_{j}}\right\rangle,\label{eq:alphaTilde}\\
    \wt{\beta}^{(j-1)}&=\eta_{z,t}\left\langle(H^{(j-1)}_{w_{j}})^{2}(A^{(j-1)}-w_{j})\right\rangle,\label{eq:betaTilde}\\
    \wt{\gamma}^{(j-1)}&=\eta_{z,t}^{2}\left\langle(H^{(j-1)}_{w_{j}})^{2}\right\rangle=\eta_{z,t}^{2}\left\langle(\wt{H}^{(j-1)}_{w_{j}})^{2}\right\rangle.\label{eq:gammaTilde}
\end{align}
The following lemma gives the concentration of quadratic forms.
\begin{lemma}
We have
\begin{align}
    \nu_{j}\left(\left\{\left|\eta_{z,t}\mbf{v}_{j}^{*}H^{(j-1)}_{w_{j}}\mbf{v}_{j}-\frac{t}{\eta_{z,t}}\wt{\alpha}^{(j-1)}\right|>\frac{\log N}{\sqrt{Nt^{2}}}\right\}\right)\leq e^{-c\log^{2}N},\label{eq:alphaConc}\\
    \nu_{j}\left(\left\{\left|\mbf{v}_{j}^{*}H^{(j-1)}_{w_{j}}(A^{(j-1)}-w_{j})\mbf{v}_{j}-\frac{t}{\eta_{z,t}}\wt{\beta}^{(j-1)}\right|>\frac{\log N}{\sqrt{Nt^{2}}}\right\}\right)\leq e^{-c\log^{2}N},\label{eq:betaConc}\\
    \nu_{j}\left(\left\{\left|\eta_{z,t}\mbf{v}_{j}^{*}\wt{H}^{(j-1)}_{w_{j}}\mbf{v}_{j}-\frac{t}{\eta_{z,t}}\wt{\gamma}^{(j-1)}\right|>\frac{\log N}{\sqrt{Nt^{3}}}\right\}\right)\leq e^{-c\log^{2}N},\label{eq:gammaConc}
\end{align}
uniformly in $\mbf{v}_{l},\,l=1,...,j-1$.
\end{lemma}
\begin{proof}
Let $B$ be a Hermitian matrix and define the moment generating function
\begin{align*}
    m_{j}\left(r;B\right)&=e^{-rt{\Tr{\wt{H}^{\left(j-1\right)}_{w_{j}}B}}}\mbb{E}_{j}\left[e^{{r}\mbf{v}_{j}^{*}B\mbf{v}_{j}}\right]\\
    &=\frac{1}{K_{j}(w_{j};A^{\left(j-1\right)})}e^{\frac{N}{t}\eta_{z,t}^{2}-rt\Tr{\wt{H}^{\left(j-1\right)}_{w_{j}}B}}\int_{S^{N-j}}e^{-\frac{N}{t}\mbf{v}_{j}^{*}\left(\eta_{z,t}^{2}+|A^{\left(j-1\right)}-w_{j}|^{2}-\frac{rt}{N}B\right)\mbf{v}_{j}}dS_{N-j}\left(\mbf{v}_{j}\right).
\end{align*}
Let
\begin{align*}
    M^{\left(j-1\right)}&=\frac{rt}{N}\sqrt{\wt{H}^{\left(j-1\right)}_{w_{j}}}B\sqrt{\wt{H}^{\left(j-1\right)}_{w_{j}}}
\end{align*}
and assume that $\|M^{\left(j-1\right)}\|<1-c$ for some $0<c<1$. Following the same steps as in the derivation of \eqref{eq:Kformula} we can rewrite this as follows:
\begin{align*}
    m_{j}\left(r;B\right)&=\frac{e^{\frac{N}{t}\eta_{z,t}^{2}}I\left(p,B\right)}{{\wt{K}_{j}}\det\left(\eta_{z,t}^{2}+|A^{\left(j-1\right)}-w_{j}|^{2}\right)}\exp\left\{-\tr M^{\left(j-1\right)}-\tr\log(1-M^{\left(j-1\right)})\right\},
\end{align*}
where
\begin{align*}
    I\left(p,B\right)&={\sqrt{\frac{N}{2\pi t^{3}}}}\int_{-\infty}^{\infty}e^{\frac{iNp}{t}}\det^{-1}\left(1+ip\sqrt{\wt{H}^{\left(j-1\right)}_{w_{j}}}(1-M^{\left(j-1\right)})^{-1}\sqrt{\wt{H}^{\left(j-1\right)}_{w_{j}}}\right)dp.
\end{align*}
Since $1-M^{\left(j-1\right)}>c$, we have
\begin{align*}
    \tr{\left(((1-M^{\left(j-1\right)})^{-1}\wt{H}^{\left(j-1\right)}_{w_{j}})^{2}\right)}&\geq\frac{N\Tr{(\wt{H}^{\left(j-1\right)}_{w_{j}})^{2}}}{(1+\|M^{\left(j-1\right)}\|)^{2}}\geq\frac{cN}{t^{3}},
\end{align*}
and so by the same argument used to prove \eqref{eq:pIntegral} we find
\begin{align*}
    I\left(p,B\right)&\leq C\sqrt{\frac{N}{t^{3}}}\int_{-\infty}^{\infty}e^{-cp^{2}\tr\left((1-M^{\left(j-1\right)})^{-1}\wt{H}^{\left(j-1\right)}_{w_{j}}\right)^{2}}dp\\
    &\leq C.
\end{align*}
From the asymptotics in {\eqref{eq:pIntegral} it then follows that
\begin{align*}
    \frac{e^{\frac{N}{t}\eta^{2}_{z,t}}I(p,B)}{\wt{K}_{j}\det(\eta^{2}_{z,t}+|A^{(j-1)}-w_{j}|^{2})}&\leq C,
\end{align*}
and so}
\begin{align}
    m_{j}\left(r;B\right)&\leq C\exp\left\{-\tr M^{\left(j-1\right)}-\tr\log(1-M^{\left(j-1\right)})\right\}\nonumber\\
    &\leq C\exp\left\{\frac{1}{1-\|M^{\left(j-1\right)}\|}\tr(M^{\left(j-1\right)})^{2}\right\},\label{eq:mBound}
\end{align}
where we have used the inequality {$-x-\log\left(1-x\right)\leq\frac{x^{2}}{1-x}$}. Thus bounding the moment generating function amounts to bounding $\left\|\sqrt{\wt{H}^{\left(j-1\right)}_{w_{j}}}B\sqrt{\wt{H}^{\left(j-1\right)}_{w_{j}}}\right\|$ and $\tr(\wt{H}^{\left(j-1\right)}_{w_{j}}B)^{2}$.

Consider first $B=\eta^{2}_{z,t}\wt{H}^{\left(j-1\right)}_{w_{j}}$; then
\begin{align*}
    \|M^{\left(j-1\right)}\|&\leq\frac{r\eta^{2}_{z,t}t}{N}\|\wt{H}^{\left(j-1\right)}_{w_{j}}\|^{2}\\
    &\leq\frac{rt}{N\eta^{2}_{z,t}}<\frac{Cr}{Nt}.
\end{align*}
{Thus we can apply the bound in \eqref{eq:mBound} up to $r=Nt/C$}. Moreover we have
\begin{align*}
    \tr(M^{\left(j-1\right)})^{2}&=\frac{r^{2}\eta_{z,t}^{4}t^{2}}{N^{2}}\tr(\wt{H}^{\left(j-1\right)}_{w_{j}})^{4}\\
    &\leq\frac{r^{2}t^{2}}{N^{2}\eta_{z,t}^{2}}\tr\wt{H}^{\left(j-1\right)}_{w_{j}}\\
    &\leq\frac{Cr^{2}}{Nt}.
\end{align*}
{By Markov's inequality we obtain
\begin{align*}
    \nu_{j}\left(\left|\eta^{2}_{z,t}\mbf{v}_{j}^{*}\wt{H}^{\left(j-1\right)}_{w_{j}}\mbf{v}_{j}-t\wt{\gamma}^{(j-1)}\right|>x\right)&\leq e^{-xr}m_{j}(r;\eta^{2}_{z,t}\wt{H}^{(j-1)}_{w_{j}}).
\end{align*}
Choosing $x=\frac{\log N}{\sqrt{Nt}}$ and $r=\sqrt{Nt}\log N$ we obtain \eqref{eq:gammaConc}.}

Now consider $B^{\left(j-1\right)}=\eta^{2}_{z,t}H^{\left(j-1\right)}_{w_{j}}$. In this case we have
\begin{align*}
    \|M^{\left(j-1\right)}\|&\leq\frac{rt}{N\eta^{2}_{z,t}}\leq\frac{Cr}{Nt},\\
    \tr(M^{\left(j-1\right)})^{2}&=\frac{r^{2}\eta_{z,t}^{4}t^{2}}{N^{2}}\tr(H^{\left(j-1\right)}_{w_{j}}\widetilde{H}^{\left(j-1\right)}_{w_{j}})^{2}\leq\frac{Cr^{2}}{N}.
\end{align*}
{Applying Markov's inequality with $x=\frac{\log N}{\sqrt{N}}$ and $r=\sqrt{N}\log N$ (note that $\sqrt{N}\log N\ll Nt$ since $Nt^{2}>N^{2\epsilon}$) we obtain \eqref{eq:alphaConc}.}
Finally we consider $B^{\left(j-1\right)}=\eta_{z,t}\Re(H^{\left(j-1\right)}_{w_{j}}(A-w_{j}))$, for which we have
\begin{align*}
    \|M^{\left(j-1\right)}\|&\leq\frac{Cr}{Nt},\\
    \tr(M^{\left(j-1\right)})^{2}&\leq\frac{Cr^{2}}{N}.
\end{align*}
We can do the same for $B^{\left(j-1\right)}=\eta_{z,t}\Im(H^{\left(j-1\right)}_{w_{j}}(A-w_{j}))$ and combine the two to obtain \eqref{eq:betaConc} {after choosing $x=\frac{\log N}{\sqrt{N}}$ and $r=\sqrt{N}\log N$ in Markov's inequality}.
\end{proof}

From this lemma it follows that
\begin{align*}
    \int_{S^{N-j}}\left|\det V_{j}^{*}G^{\left(j-1\right)}_{z}V_{j}\right|^{j}d\nu_{j}\left(\mbf{v}_{j}\right)&=\left[1+O\left(\frac{\log N}{\sqrt{Nt^{2}}}\right)\right]\left(\frac{t^{2}(\wt{\alpha}^{(j-1)}\wt{\gamma}^{(j-1)}+|\wt{\beta}^{(j-1)}|^{2})}{\eta^{2}_{z,t}}\right)^{j}.
\end{align*}
Therefore \Cref{lem:nuExpectation} will follow from
\begin{lemma}
For $j=1,...,k$ we have
\begin{align}
    \left|\wt{\alpha}^{(j-1)}-\alpha_{z,t}\right|&\leq\frac{C}{\sqrt{Nt^{3}}},\label{eq:alphaDiff}\\
    \left|\wt{\beta}^{(j-1)}-\beta_{z,t}\right|&\leq\frac{C}{\sqrt{Nt^{3}}},\label{eq:betaDiff}\\
    \left|\wt{\gamma}^{(j-1)}-\gamma_{z,t}\right|&\leq\frac{C}{\sqrt{Nt^{3}}},\label{eq:gammaDiff}
\end{align}
uniformly in $\mbf{v}_{l},\,l=1,...,j-1$.
\end{lemma}
\begin{proof}
We can use interlacing or \Cref{lem:minorresolvent} to replace $A^{\left(j-1\right)}$ with $A$ and then use the resolvent identity to replace $w_{j}$ with $z$. For $\wt{\gamma}^{(j-1)}$ we find
\begin{align*}
    \wt{\gamma}^{(j-1)}&=\frac{1}{2}\left\langle\Im G^{\left(j-1\right)}_{w_{j}})^{2}\right\rangle\\
    &=\frac{1}{2\eta_{z,t}}\Im\left\langle G_{w_{j}}\right\rangle-\frac{1}{2}\Re\left\langle G^{2}_{w_{j}}\right\rangle+O\left(\frac{1}{t^{2}}\right)\\
    &=\frac{1}{2}\left\langle(\Im G_{z})^{2}\right\rangle+\frac{1}{2\sqrt{N}\eta_{z,t}}\Im\left\langle G_{z}Z_{j}G_{w_{j}}\right\rangle\\&-\frac{1}{2}\Re\left\langle\frac{2}{\sqrt{N}}G^{2}_{z}Z_{j}G_{w_{j}}+\frac{1}{N}(G_{z}Z_{j}G_{w})^{2}\right\rangle+O\left(\frac{1}{t^{2}}\right)\\
    &=\gamma_{z,t}+O\left(\frac{1}{\sqrt{Nt^{3}}}\right).
\end{align*}
For $\wt{\alpha}^{(j-1)}$ we find
\begin{align*}
    \wt{\alpha}^{(j-1)}&=-\langle G^{(j-1)}_{w_{j}}E^{(j)}G^{(j-1)}_{w_{j}}E^{(j)*}\rangle\\
    &=-\langle G_{w_{j}}E G_{w_{j}}E^{*}\rangle+O\left(\frac{1}{Nt^{2}}\right)\\
    &=-\langle G_{z}E G_{z}E^{*}\rangle-\frac{1}{\sqrt{NE_{z,t}}}\left\langle G_{w_{j}}Z_{j}G_{z}E G_{z}E^{*}+G_{z}E G_{w_{j}}Z_{j}G_{z}E^{*}\right\rangle\\
    &-\frac{1}{NE_{z,t}}\left\langle G_{w_{j}}Z_{j}G_{z}E G_{w_{j}}Z_{j}G_{z}E^{*}\right\rangle+O\left(\frac{1}{Nt^{2}}\right)\\
    &=\alpha_{z,t}+O\left(\frac{1}{\sqrt{Nt^{3}}}\right),
\end{align*}
and we have the lower bound $\alpha_{z,t}\geq c$ from \eqref{eq:A_alpha}. For $\wt{\beta}^{(j-1)}$, \eqref{eq:betaDiff} follows by the same steps.
\end{proof}
As a result of this lemma and the lower bound $\gamma_{z,t}\geq C/t$ (which follows directly from \eqref{eq:A_gamma}) we find
\begin{align*}
    \wt{\sigma}^{(j-1)}&:=\wt{\alpha}^{(j-1)}+\frac{|\wt{\beta}^{(j-1)}|^{2}}{\wt{\gamma}^{(j-1)}}\\
    &=\alpha_{z,t}+\frac{|\beta_{z,t}|^{2}}{\gamma_{z,t}}+O\left(\frac{1}{\sqrt{Nt^{3}}}\right)\\
    &=\sigma_{z,t}+O\left(\frac{1}{\sqrt{Nt^{3}}}\right).
\end{align*}
{Thus we have
\begin{align*}
    \left(\frac{\eta^{2}_{z,t}}{t^{2}\gamma_{z,t}\sigma_{z,t}}\right)^{i}\int_{S^{N-i}}|\det V_{i}^{*}G^{(i-1)}_{z}V_{i}|^{i}d\nu_{i}(\mbf{v}_{i})&=\left(\frac{\eta^{2}_{z,t}}{t^{2}\gamma_{z,t}\sigma_{z,t}}\right)^{i}\left(\frac{t^{2}\wt{\gamma}^{(i-1)}\wt{\sigma}^{(i-1)}}{\eta_{z,t}^{2}}\right)^{i}\\
    &=\left(\frac{\wt{\gamma}^{(i-1)}\wt{\sigma}^{(i-1)}}{\gamma_{z,t}\sigma_{z,t}}\right)^{i}\\
    &=1+O\left(\frac{1}{\sqrt{Nt^{3}}}\right),
\end{align*}}
which proves \eqref{eq:qformexpectation}. Note that if we assume \eqref{eq:A3} then the error $|\wt{\alpha}^{(j-1)}-\alpha_{z,t}|$ and hence the error $|\wt{\sigma}^{(j-1)}-\sigma_{z,t}|$ becomes $O\left(1/\sqrt{Nt^{2}}\right)$.

\section{Proof of \Cref{thm2}}
To prove \Cref{thm2}, we first need to show that if $A$ has i.i.d. complex entries with mean zero, variance one and finite moments then it satisfies the assumptions of \Cref{thm1}. Let
\begin{align}
     M^{\left(1\right)}_{z}\left(\eta\right)&=\begin{pmatrix} m_{z}\left(\eta\right)& -zu_{z}\left(\eta\right)\\ -\bar{z}u_{z}\left(\eta\right)&m_{z}\left(\eta\right)\end{pmatrix},\label{eq:Mz}
\end{align}
where $m_{z}$ is positive imaginary solution to the cubic equation
\begin{align}
    -\frac{1}{m_{z}}&=i\eta+m_{z}-\frac{|z|^{2}}{i\eta+m_{z}},
\end{align}
and
\begin{align}
    u_{z}\left(\eta\right)&=\frac{m_{z}\left(\eta\right)}{i\eta+m_{z}\left(\eta\right)}.
\end{align}
The single resolvent local law states that, on a set of high probability, we can replace $G_{z}$ by $M^{\left(1\right)}_{z}$. We will use the statement of Alt, Erd\H{o}s and Kr\"{u}ger \cite{alt_spectral_2021}.
\begin{prop}[Averaged single resolvent local law (Theorem 5.2 in \cite{alt_spectral_2021})]\label{prop:singlell}
Let $z\in\mbb{C}$, $B$ be deterministic, $\epsilon>0$ and $N^{-1+\epsilon}<\eta<N^{100}$. Then for any $\sigma>0$ and $D>0$ we have
\begin{align}
    \textup{Pr}\left(\Big|\left\langle(G_{z}\left(\eta\right)- M_{z}\left(\eta\right))B\right\rangle\Big|>\frac{N^{\sigma}\|B\|}{N\eta}\right)&\leq N^{-D},
\end{align}
for $N>N\left(\sigma,D\right)$.
\end{prop}

We also need to consider expressions of the form $\left\langle G_{z}(\eta_{1})B_{1}G_{z}(\eta_{2})B_{2}\right\rangle$. Define the linear operators $S$ and $B_{z}(\eta_{1},\eta_{2})$ on $\mbf{M}_{2N}$ by
\begin{align}
    S\left[\begin{pmatrix}X&Y\\ W&Z\end{pmatrix}\right]&=\begin{pmatrix}\left\langle Z\right\rangle&0\\0&\left\langle X\right\rangle\end{pmatrix},\\
    B_{z}\left(\eta_{1},\eta_{2}\right)(\mc{X})&=1-M^{\left(1\right)}_{z}\left(\eta_{1}\right)S[\mc{X}]M^{\left(1\right)}_{z}\left(\eta_{2}\right),
\end{align}
and let
\begin{align}
    M^{\left(2\right)}_{z}\left(\eta_{1},\eta_{2};\mc{X}\right)&=\left[B_{z}\left(\eta_{1},\eta_{2}\right)\right]^{-1}[M^{\left(1\right)}_{z}\left(\eta_{1}\right)\mc{X}M^{\left(1\right)}_{z}\left(\eta_{2}\right)].
\end{align}
The two-resolvent local law has recently been proven by Cipolloni, Erd\H{o}s and Schr\"{o}der \cite{cipolloni_mesoscopic_2023}. We state here a special case of their result:
\begin{prop}[Averaged two-resolvent local law (Theorem 3.2 in \cite{cipolloni_mesoscopic_2023})]\label{prop:doublell}
Let $z\in\mbb{C}$, $B_{j}$, $j=1,2$, be deterministic, $\epsilon>0$ and $N^{-1+\epsilon}<\eta<N^{100}$. Then for any $\sigma>0$ and $D>0$ we have
\begin{align}
    \textup{Pr}\left(\Big|\left\langle(G_{z}\left(\eta_{1}\right)B_{1} G_{z}\left(\eta_{2}\right)- M^{\left(2\right)}_{z}\left(\eta_{1},\eta_{2};B_{1}\right)B_{2}\right\rangle\Big|>\frac{N^{\sigma}\|B_{1}\|\cdot\|B_{2}\|}{N\eta^{2}}\right)&\leq N^{-D},
\end{align}
for $N>N\left(\sigma,D\right)$.
\end{prop}
Theorem 3.2 in \cite{cipolloni_mesoscopic_2023} allows for resolvents $ G_{z_{j}}$ at two different $z_{1},z_{2}$ and any determinisitic $B_{1},B_{2}$, but we only need to consider $z_{1}=z_{2}$ and $B_{1},B_{2}\in\{1_{2N},E,E^{*}\}$. Both \Cref{prop:singlell} and \Cref{prop:doublell} can be extended to hold uniformly in $z$ in any compact subset of $\mbb{C}$ by taking a union bound over a sufficiently fine net, using the fact that $z\mapsto G_{z}(\eta)$ is Lipschitz with constant bounded by $\eta^{-2}$.

In the following lemma we collect some properties of $ M^{\left(1\right)}_{z}$ and $ M^{\left(2\right)}_{z}$.
\begin{lemma}\label{lem:M1M2}
When $|z|<1$, we have
\begin{align}
    m_{z}\left(\eta\right)&=i\left[\sqrt{1-|z|^{2}}+\frac{2|z|^{2}-1}{1-|z|^{2}}\eta+O\left(\eta^{2}\right)\right],\label{eq:masymp}\\
    u_{z}\left(\eta\right)&=1-\frac{\eta}{\sqrt{1-|z|^{2}}}+O\left(\eta^{2}\right)\label{eq:uasymp},
\end{align}
and
\begin{align}
    \sigma_{z}\left(\eta\right)&=1+O\left(\eta\right).\label{eq:sigmaasymp}
\end{align}
Moreover, if $|z|<1-\omega$ there is a $C>0$ depending on $\omega$ such that
\begin{align}
    \left\|M^{(2)}_{z}(\eta_{1},\eta_{2};E)\right\|&\leq C,\label{eq:M2Norm}
\end{align}
where we recall the definition of $E$ in \eqref{eq:rho}.
\end{lemma}
\begin{proof}
Eq. \eqref{eq:masymp} and \eqref{eq:uasymp} follow by applying perturbation theory to the cubic equation defining $m_{z}\left(\eta\right)$. To calculate $\sigma_{z}\left(\eta\right)$ we need to analyse $ M^{\left(2\right)}_{z}\left(\eta_{1},\eta_{2};B\right)$, for which we follow Lemma 6.1 in \cite{cipolloni_central_2023}. We need only consider the invariant subspace consisting of $2\times2$ block matrices whose blocks are multiples of the identity. If we map this to $\mbb{C}^{4}$ by the identification
\begin{align*}
    \begin{pmatrix}a&b\\c&d\end{pmatrix}&\mapsto\left(a\,d\,b\,c\right)^{T},
\end{align*}
then $B_{z}\left(\eta_{1},\eta_{2}\right)$ has the matrix representation
\begin{align*}
    B_{z}\left(\eta_{1},\eta_{2}\right)&=\begin{pmatrix}T_{1}&0\\T_{2}&1\end{pmatrix},
\end{align*}
where, with $m_{j}:=m_{z}\left(\eta_{j}\right)$ and $u_{j}:=u_{z}\left(\eta_{j}\right)$,
\begin{align*}
    T_{1}&=\begin{pmatrix}1-|z|^{2}u_{1}u_{2}&-m_{1}m_{2}\\-m_{1}m_{2}&1-|z|^{2}u_{1}u_{2}\end{pmatrix},
\end{align*}
and
\begin{align*}
    T_{2}&=\begin{pmatrix}zm_{2}u_{1}&zm_{1}u_{2}\\\bar{z}m_{1}u_{2}&\bar{z}m_{2}u_{1}\end{pmatrix}.
\end{align*}
The inverse of $B_{z}$ is now easily calculated:
\begin{align*}
    [B_{z}\left(\eta_{1},\eta_{2}\right)]^{-1}&=\begin{pmatrix}T_{1}^{-1}&0\\-T_{2}T_{1}^{-1}&1\end{pmatrix}.
\end{align*}
Observe that $T^{-1}_{1}$ has two eigenvalues $\lambda_{\pm}=(1-|z|^{2}u_{1}u_{2}\pm m_{1}m_{2})^{-1}$ with corresponding eigenvectors $(1,\mp1)^{T}$. Thus $[B_{z}(\eta_{1},\eta_{2})]^{-1}$ has two eigenvalues $\lambda_{\pm}$ with eigenvectors $E_{+}\mp E_{-}$ (recall the definitions in \eqref{eq:rhoPlus} and \eqref{eq:rhoMinus}); the remaining eigenvalues are all equal to 1. From the asymptotics in \eqref{eq:masymp} and \eqref{eq:uasymp}, we have $\lambda_{-}<C$ when $|z|<1-\omega$ for some $\omega>0$ and $\lambda_{+}<C(\eta_{1}+\eta_{2})^{-1}$. Now note that $|u_{1}-u_{2}|\leq C(\eta_{1}+\eta_{2})$ and $|m_{1}-m_{2}|\leq C(\eta_{1}+\eta_{2})$ and so
\begin{align*}
    \left|\tr((E_{+}-E_{-})M^{(1)}_{z}(\eta_{1})EM^{(1)}_{z}(\eta_{2}))\right|&=|z(m_{1}u_{2}-m_{2}u_{1})|\leq C(\eta_{1}+\eta_{2}),
\end{align*}
i.e. the component of $M^{(1)}_{z}(\eta_{1})EM^{(2)}_{z}(\eta_{2})$ in the direction of the eigenvector of $[B_{z}(\eta_{1},\eta_{2})]^{-1}$ corresponding to the largest eigenvalue $\lambda_{+}$ is $O(\eta_{1}+\eta_{2})$. This implies \eqref{eq:M2Norm}.

With these considerations and the asymptotics in \eqref{eq:masymp} and \eqref{eq:uasymp}, we can evaluate traces of the form $\tr G_{z}\left(\eta_{1}\right)B_{1} G_{z}\left(\eta_{2}\right)B_{2}$. We will find
\begin{align*}
    \alpha_{z}\left(\eta\right)&=1+O\left(\eta\right),\\
    \beta_{z}\left(\eta\right)&=O(1),\\
    \gamma_{z}\left(\eta\right)&=O(\eta^{-1}),
\end{align*}
and hence
\begin{align*}
    \sigma_{z}\left(\eta\right)&=\alpha^{2}_{z}\left(\eta\right)+\frac{|\beta_{z}\left(\eta\right)|^{2}}{\gamma^{2}_{z}\left(\eta\right)}\\
    &=1+O\left(\eta\right).
\end{align*}
\end{proof}

From this it follows immediately that, if $\mc{E}$ denotes the event that \eqref{eq:A_g}, \eqref{eq:A_alpha}, \eqref{eq:A_beta}, \eqref{eq:A_gamma} and \eqref{eq:A2} hold, then $P\left(\mc{E}\right)\geq1-N^{-D}$ for any $D>0$. Now let $t=N^{-1/3+\epsilon}$, construct the Gauss-divisible matrix $A+\sqrt{t}B$ and define the statistic
\begin{align}
    S_{N,z}\left(f\right)&=\sum_{i_{1}\neq\cdots\neq i_{k}}f\left(\sqrt{N}\left(z-z_{1}\right),...,\sqrt{N}\left(z-z_{k}\right)\right).
\end{align}
Then since $|S_{N,z}|<N^{k}$ we have
\begin{align*}
    \mbb{E}\left[S_{N,z}\left(f\right)\right]&=\mbb{E}[1_{\mc{E}}S_{N,z}\left(f\right)]+O(N^{k-D})\\
    &=\mbb{E}[\mbb{E}[S_{N,z}\left(f\right)|A\in\mc{E}]]+O\left(N^{k-D}\right).
\end{align*}
The error term can be neglected by choosing $D>k$. By \Cref{thm1} and \eqref{eq:sigmaasymp} we have
\begin{align*}
    \mbb{E}[S_{N,z}\left(f\right)|A\in\mc{E}]&=\frac{1}{N^{k}}\int_{\mbb{C}^{k}}f\left(\mbf{z}\right)\rho^{\left(k\right)}_{N}\left(z+\frac{\mbf{z}}{\sqrt{N}};A\right)d\mbf{z}\\
    &=\frac{1}{\left(N\sigma_{z,t}\right)^{k}}\int_{\mbb{C}^{k}}f\left(\sqrt{\sigma_{z,t}}\mbf{z}\right)\rho^{\left(k\right)}_{N}\left(z+\frac{\mbf{z}}{\sqrt{N\sigma_{z,t}}};A\right)d\mbf{z}\\
    &\to\int_{\mbb{C}^{k}}f\left(\mbf{z}\right)\rho^{\left(k\right)}_{GinUE}\left(\mbf{z}\right)d\mbf{z},
\end{align*}
and so the same limit holds for $\mbb{E}[S_{N,z}\left(f\right)]$. Now we use the four moment theorem (Theorem 2 in \cite{tao_random_2015}), which states that if $A$ and $\wt{A}$ have i.i.d. entries with finite moments that match up to order four then the correlation functions tend to the same limit. The statement in \cite{tao_random_2015} requires subexponential entries and both $A$ and $\wt{A}$ to match moments to third order with a Gaussian matrix, but this is due to the use of a local law proven in the same paper (Theorem 20) under these conditions. However, as noted in Remarks 3 and 22 of \cite{tao_random_2015}, we can replace this by the local law of \cite{bourgade_local_2014}, which does not require these conditions. Moreover, upon inspection of the proof one observes that the moment matching condition can be relaxed (as was done for Hermitian matrices in \cite{erdos_bulk_2010-1}) to allow for an error of $O(N^{-1/2-\epsilon})$ for the third moment and $O(N^{-\epsilon})$ for the fourth moment, for any $\epsilon>0$. We choose an $\wt{A}$ such that: i) it has i.i.d. entries with finite moments; ii) $A$ and $\frac{1}{\sqrt{1+t}}\left(\wt{A}+\sqrt{Nt}B\right)$ match moments up to third order {(note that here we have $\sqrt{N}B$ since the entries of $B$ have variance $N^{-1}$)}; iii) the fourth moment differs by $O(t)=O(N^{-1/3+\epsilon})$. The existence of such an $\wt{A}$ follows by Lemma 3.4 in \cite{erdos_universality_2011-2}.

\paragraph{\small Acknowledgements}
\small{We thank Benjamin Landon for pointing out that the conditions in Theorem 1.2 could be improved using Lemma 3.4 in \cite{erdos_universality_2011-2}. We also thank Dang-Zheng Liu for pointing out the references \cite{grela_diffusion_2016} and \cite{liu_phase_2022}. We gratefully acknowledge the support of the Royal Society (grant numbers RF/ERE/210051 and URF/R/221017).}

\bibliographystyle{plain}
\bibliography{revision}

\end{document}